\documentclass{amsart}

\usepackage{amsmath,amsthm,amssymb,amsfonts,amsopn}
\usepackage{graphicx,epstopdf}
\usepackage[colorlinks=true,linkcolor=black,citecolor=black,urlcolor=black,filecolor=black]{hyperref}
\usepackage[nameinlink]{cleveref}
\usepackage{doi}
\usepackage{xcolor,lipsum}
\usepackage{chngcntr}

\usepackage[
  backend=biber,
  style=numeric-comp,
  sorting=nyt,
  sortcites=true,
  giveninits=true, 
  doi=true, url=false, isbn=false
]{biblatex}

\DeclareFieldFormat{doi}{\mkbibacro{DOI}\addcolon\space\href{DOI: #1}{#1}}

\DeclareFieldFormat{labelnumber}{\mkbibbold{#1}}

\DeclareDelimFormat{multicitedelim}{%
  \addnbspace\bibclosebracket\addnbspace\bibopenbracket}

\DeclareFieldFormat{labelnumberwidth}{\mkbibbold{#1}}

\title[Forward Euler Wasserstein Gradient Flows]%
{Forward Euler for Wasserstein Gradient Flows: Breakdown and Regularization}

\author{Yewei Xu}
\address{Department of Mathematics, University of Wisconsin-Madison, Madison, WI}
\email{xu464@wisc.edu}

\author{Qin Li}
\address{Department of Mathematics, University of Wisconsin-Madison, Madison, WI}
\email{qinli@math.wisc.edu}

\thanks{This work was partially funded by NSF-DMS-2308440.}
\thanks{*Corresponding author: Yewei Xu (\href{mailto:xu464@wisc.edu}{xu464@wisc.edu}).}

\newcommand{\rd}{\mathrm{d}}
\newcommand{\KL}{\mathrm{KL}}
\newcommand{\Hess}{\mathrm{Hess}}

\newcommand{\J}{\mathrm{J}}

\numberwithin{equation}{section}

\theoremstyle{plain}
\newtheorem{theorem}{Theorem}[section]
\newtheorem{lemma}[theorem]{Lemma}
\newtheorem{proposition}[theorem]{Proposition}
\newtheorem{corollary}[theorem]{Corollary}

\newtheorem{assumption}[theorem]{Assumption}

\theoremstyle{definition}
\newtheorem{definition}[theorem]{Definition}

\theoremstyle{remark}
\newtheorem{remark}[theorem]{Remark}

\newtheorem*{definition*}{Definition}
\newtheorem*{remark*}{Remark}
\newtheorem*{example*}{Example}

\crefname{theorem}{Theorem}{Theorems}
\crefname{lemma}{Lemma}{Lemmas}
\crefname{proposition}{Proposition}{Propositions}
\crefname{corollary}{Corollary}{Corollaries}
\crefname{claim}{Claim}{Claims}
\crefname{assumption}{Assumption}{Assumptions}
\crefname{hypothesis}{Hypothesis}{Hypotheses}
\crefname{definition}{Definition}{Definitions}
\crefname{example}{Example}{Examples}
\crefname{fact}{Fact}{Facts}
\crefname{remark}{Remark}{Remarks}

\begin{document}

\begin{abstract}
Wasserstein gradient flows have become a central tool for optimization problems over probability measures. A natural numerical approach is forward-Euler time discretization. We show, however, that even in the simple case where the energy functional is the Kullback-Leibler (KL) divergence against a smooth target density, forward-Euler can fail dramatically: the scheme does not converge to the gradient flow, despite the fact that the first variation $\nabla\frac{\delta F}{\delta\rho}$ remains formally well defined at every step. We identify the root cause as a loss of regularity induced by the discretization, and prove that a suitable regularization of the functional restores the necessary smoothness, making forward-Euler a viable solver that converges in discrete time to the global minimizer.
\end{abstract}



\maketitle

\section{Introduction}\label{sec:intro}

Minimizing a functional over the space of probability measures has attracted significant attention in recent years~\cite{BUDAPK24}, \cite{CLTW25}, \cite{SKL20}. Many practical optimization problems naturally arise in this setting and can be formulated as
\begin{equation}\label{eqn:min_E}
\rho_{\mathrm{opt}}=\operatorname{argmin}_{\rho\in\mathcal{P}(\Omega)} F[\rho]\,,
\end{equation}
where $\mathcal{P}(\Omega)$ denotes the set of probability measures on a domain $\Omega\subset\mathbb{R}^{d}$, and $F$ is a functional mapping a measure to a real number. This formulation extends classical optimization in Euclidean space. By analogy with gradient descent in $\mathbb{R}^d$, one may consider the associated gradient flow (GF) PDE
\begin{equation}\label{eqn:GF}
\partial_t\rho_t = -\nabla_{\mathfrak{m}}F[\rho_t]
= \nabla\cdot\Bigl(\rho_t\nabla\left.\frac{\delta F}{\delta\rho}\right|_{\rho_t}\Bigr)\,,
\end{equation}
where $\mathfrak{m}$ denotes the chosen metric on $\mathcal{P}$. Restricting to $\mathcal{P}_2$, the class of probability measures with finite second moment, the Wasserstein $W_2$ metric provides a natural choice and yields an explicit expression for the gradient. The subscript $t$ explicitly expresses the time dependence of $\rho$, which remains a probability measure at all times.

As derived in~\cite{AGS08}, when $F$ is Wasserstein differentiable (W-differentiable), the Wasserstein gradient (W-gradient) takes the form $-\nabla\frac{\delta F}{\delta\rho}$, a $d$-dimensional vector field prescribing the velocity of the evolving measure, as seen in the second equation in~\eqref{eqn:GF}. In direct analogy with the ODE $\dot{x}=-\nabla_x f(x)$, whose long-time limit seeks $x_{\mathrm{opt}}=\arg\min f(x)$, the aim here is to run~\eqref{eqn:GF} so that $\rho_t$ converges, in $W_2$, to the minimizer $\rho_{\mathrm{opt}}$ in the long-time limit.

Time discretization for this PDE appears straightforward. The simplest method is forward Euler (FE), which yields
\begin{equation}\label{eqn:GF_dis}
\rho_{n+1}=\left(T_n\right)_\#\rho_n\,,\quad\text{with}\quad T_n(x) = x-h_n\nabla\left.\frac{\delta F}{\delta\rho}\right|_{\rho_n}(x)\,.
\end{equation}
Here $T:\mathbb{R}^d\to\mathbb{R}^d$ denotes the push-forward map defined by one FE step with stepsize $h_n$, and we use $t$ and $n$ to distinguish continuous and discrete indices, respectively.

This PDE also admits a natural particle interpretation that facilitates in-space discretization. If $X_0\sim\rho_0$ is sampled from the initial distribution and evolves under the velocity field $-\nabla\frac{\delta F}{\delta\rho}$, then $X_t$ remains a sample from $\rho_t$:
\begin{equation}\label{eqn:GF_particle}
\text{Let}\quad \dot X_t = -\nabla\left.\frac{\delta F}{\delta\rho}\right|_{\rho_t}(X_t),
\qquad\text{then}\qquad
\operatorname{Law}(X_t)=\rho_t\;,\qquad\forall t\,.
\end{equation}

Therefore, the fully discrete form reads
\begin{equation}
X_0\sim\rho_0,\quad X_{n+1}=T_n(X_n)\,.
\end{equation}
This scheme is intuitive, easy to implement, and widely used in engineering applications (see e.g.~\cite{AF21}, \cite{HTT24}, \cite{JLWYZ24}, \cite{NNSSR24}, \cite{WCL22}, \cite{YENM23}, \cite{YZL23}). But it prompts the central question of this paper:

\medskip
\begin{center}
\emph{Is FE a valid in-time discretization for solving gradient flow?}
\end{center}

\medskip
Or in mathematical terms,

\medskip
\begin{center}
\textbf{Main Question:}\quad \emph{Does~\eqref{eqn:GF_dis} faithfully approximate~\eqref{eqn:GF}?}
\end{center}

\medskip
Despite its popularity and apparent naturalness, our conclusion is negative. We demonstrate that FE is structurally incompatible with gradient flow dynamics:
\begin{enumerate}
\item The failure is not merely quantitative (e.g. slow convergence), but qualitative: we exhibit concrete examples where the FE trajectory remains a fixed distance away from the true gradient flow, regardless of step-size tuning;
\item The failure is also generic. For initial data not in $C^\infty$, FE necessarily breaks down within finitely many iterations.
\end{enumerate}

This negative result is striking, especially given the simplicity of FE. Yet through our analysis, it is clear to us that the mechanism of the breakdown is triggered by a loss of regularity. This recognition suggests a remedy—regularizing the objective functional. We propose a modification: The modification is small, so the continuous gradient flow is largely unchanged, but it restores sufficient regularity to make FE once again a solver that approximates GF and thus finding optimum.

\subsection{Brief overview}
We leave the details to later sections but discuss the core findings briefly in this section. The major difficulty emerges when $F$ lacks sufficient regularity across $\mathcal{P}_2$. A natural and instructive example is the celebrated Kullback-Leibler (KL) divergence, sometimes referred as relative entropy in physics, one of the most widely used functionals on $\mathcal{P}_2$. Its analytical structure allows us to make the phenomenon fully explicit.

Suppose the target distribution has smooth density $\rho^*=e^{-U}$, so that $F[\rho]=\KL[\rho|\rho^*]$. When $U$ is strongly convex, $F$ is geodesically convex in $\mathcal{P}_2$, with a unique minimizer $\rho_{\mathrm{opt}}=\rho^*$ satisfying $F_{\mathrm{opt}}=F[\rho_{\mathrm{opt}}]=F[\rho^*]=0$. Assuming $F$ is W-differentiable at $\rho_t$, one can compute the W-derivative explicitly. Substituting in~\eqref{eqn:GF}, we obtain
\begin{equation}\label{eqn:FP}
\partial_t\rho_t-\nabla\cdot(\rho_t\nabla U)=\Delta\rho_t\,,
\end{equation}
the classical Fokker-Planck equation, a cornerstone of statistical mechanics and modern machine learning. This PDE is of convection-diffusion type, and the diffusion term enforces instantaneous smoothing, ensuring that solutions remain highly regular for all time.

Ideally, one would like to preserve this regularity in the discrete setting. However, as shown in Section~\ref{sec:counterexamples}, the FE scheme artificially destroys regularity: each step consumes two derivatives, and the accumulated loss rapidly destabilizes the method. This degradation is purely a discrete-time artifact. After finitely many steps, the numerical trajectory becomes much rougher than its PDE counterpart.

Such breakdowns, caused by mismatched time-stepping and spatial discretization, are familiar in PDE numerics. Classical solvers (finite difference, finite element, spectral) detect loss of regularity when derivatives cannot be computed, effectively triggering an internal stability check. For gradient flows of the form~\eqref{eqn:GF}, however, this safeguard is absent. Since $\rho$ evolves under the W-derivative, the scheme continues to update via~\eqref{eqn:GF_dis}, treating $\nabla\frac{\delta F}{\delta\rho}$ as the velocity field. Yet for many $\rho\in\mathcal{P}_2$, the functional $F$ is not W-differentiable, even though $\nabla\frac{\delta F}{\delta\rho}$ appears to exists and is formally computable. This superficial computability enables FE to proceed without warning, even after the underlying gradient flow structure has collapsed.

This reflects a fundamental distinction between probability measure spaces and classical Banach spaces. In Banach spaces, the definition of regularity and derivatives is uniform across the space. In contrast, in $\mathcal{P}_2$, regularity is metric-dependent, and notions of differentiability differ. In Section~\ref{sec:prelim}, we review three notions: first variation, W-differentiability, and Lions differentiability (L-differentiability). They coincide for sufficiently smooth $F$, but diverge in general. In particular, the expression $\nabla\frac{\delta F}{\delta\rho}$ may remain explicit even when $F$ is not W-differentiable. Building on this, Section~\ref{sec:counterexamples} provides explicit constructions where such discrepancies directly cause FE breakdown.

These insights point to the heart of the issue: the KL functional enjoys full regularity only on a restricted subdomain of $\mathcal{P}_2$. Once solutions leave this region, FE fails. The natural resolution is therefore to regularize $F$, so that the regularized version enjoys differentiability over the full space. In Section~\ref{sec:regularized_KL}, we introduce this regularized KL functional $F^\epsilon$ and show that, on any bounded convex subset $\mathcal{C}\subset\mathbb{R}^d$, its Wasserstein gradient always exists and retains the form $\nabla\frac{\delta F^\epsilon}{\delta\rho}$. This restores compatibility of FE with the underlying gradient flow and guarantees convergence to the optimization problem.

\subsection{Related works}\label{sec:related_works}

The study of gradient flows in the space of probability measures has attracted significant attention in recent years, largely due to its broad applicability, particularly in machine learning. Foundational groundwork was laid in~\cite{AGS08}, \cite{Villani2003}, \cite{Villani2009}, and several recent advances can be found in~\cite{CCWY22}, \cite{CLTW25}, \cite{CLCS25}, \cite{LLW20}, \cite{MKLGSB21}, \cite{NR23}, \cite{SKL20}, \cite{TM19}.

From a numerical perspective, simulating such flows typically relies on implicit time discretizations. The most celebrated approach in this direction is the Jordan-Kinderlehrer-Otto (JKO) scheme~\cite{JKO98}, which can be interpreted as a backward Euler method in the Wasserstein metric. Despite its well-established $\Gamma$-convergence and stability properties~\cite{AGS08}, the implicit nature of the JKO scheme makes it computationally demanding, as each step requires solving a full variational problem.

A line of research has also focused on developing explicit solvers. A particularly general framework was introduced by Cavagnari, Savar\'e, and Sodini~\cite{CSS23P}, \cite{CSS23A}, \cite{CSS25}, who, under $\lambda$-dissipativity assumptions on the vector field, proved well-posedness and established a strong convergence rate of order $O(\sqrt{h})$ for forward Euler schemes in the Wasserstein space.

Our work on providing counter-example can be viewed as a complement to~\cite{CSS25}. 
In particular, we observe that even though the Kullback-Leibler (KL) functional is geodesically convex and formally differentiable, it does not induce a $\lambda$-dissipative vector field as required in~\cite{CSS25}. 
Consequently, their convergence results do not apply in this setting; see Section~\ref{sec:counterexamples} for details.

Finally, the rise of ensemble methods as alternative optimization solvers has motivated interacting particle formulations for solving~\eqref{eqn:GF_dis}~\cite{GHLS20}, \cite{MRO20}, \cite{NR23}, \cite{SWZ25}. 
This approach naturally requires density estimation, which in turn led to the development of the blob method~\cite{Carrillo2019}, \cite{CJT25}. 
The latter is closely related to the regularized Kullback-Leibler functional introduced below in Section~\ref{sec:regularized_KL}.

\section{Basics and notations for gradient flow}\label{sec:prelim}

We begin with a brief review of the basic structures underlying gradient flow on probability measure spaces.

The optimization problem~\eqref{eqn:min_E} can be regarded as an extension of classical optimization in Euclidean space, where one seeks $\min_{x\in\mathbb{R}^d}f(x)$. A standard method is gradient descent, which evolves according to $\dot{x}=-\nabla f$ so that $\lim_{t\to\infty}f(x(t))=\mathrm{min}(f)$. With additional assumptions on $f$, one can derive non-asymptotic convergence rates as well. 

Translating this strategy to the problem~\eqref{eqn:min_E} is not straightforward. Two fundamental challenges arise:
\begin{itemize}
    \item \textbf{Infinite dimensionality.} In the classical setting, the variable $x$ lies in finite-dimensional space $\mathbb{R}^d$. In contrast, $\rho\in\mathcal{P}(\Omega)$ is an infinite-dimensional object. Moving $\rho$ requires the language of partial differential equations, which describe the time evolution of $\rho_t(\cdot)$ over $\mathcal{P}(\Omega)$.
    \item \textbf{Nonlinearity of the space.} Unlike Euclidean or Hilbert spaces, which admit global linear structures and norms, $\mathcal{P}(\Omega)$ is nonlinear and must be treated as a manifold. Without a suitable metric, even the notion of a gradient is ill-defined. Because of this nonlinearity, distances between probability measures must be defined locally. Several metrics are possible~\cite{A16}, \cite{PW24}, but the Wasserstein distance has become the most widely adopted~\cite{BJGR19}, \cite{FZMAP15}, \cite{GK23}, \cite{SQZY18}. Arising from optimal transport, it measures the length of geodesics between probability measures. Equipped with this metric, one can define tangent spaces and gradients in a well-posed manner.
\end{itemize}

Adopting the Wasserstein metric, the natural analogue of gradient descent is gradient flow, leading to the formulation in~\eqref{eqn:GF}.

\subsection{Basic Notations and Wasserstein differentiability}\label{sec:basic_def}
We begin by collecting the notations and assumptions used throughout the paper.
\begin{itemize}
    \item[--]\textbf{Probability measures.}
    $\mathcal{P}_2(\Omega)$ denotes the set of probability measures with finite second moment on $\Omega$.
    \item[--]\textbf{Functionals.}     $F:\mathcal{P}_2(\Omega)\to\mathbb{R}\cup\{\infty\}$ is a functional, with effective domain:
    \begin{equation}\label{eqn:D_F}
        D(F)=\{\mu\in\mathcal P_2(\Omega):\,F[\mu]<\infty\}\,.
        \end{equation}
    \item[--] \textbf{Absolutely continuous measures.}    
    $\mathcal{P}_2^r(\Omega)$ is the subset of $\mathcal{P}_2(\Omega)$ consisting of measures absolutely continuous with respect to Lebesgue. For $\rho\in \mathcal{P}_2^r(\Omega)$ we write $\rd\rho=p\rd x$, and the density has the form $p=\exp(-U)$ for some $U:\Omega\to\mathbb{R}\cup\{\infty\}$. When the context is clear, we use $\rho$ both for the measure and its density. The measure $\rho$ is called log-concave if $U$ is convex.
    \item[--] \textbf{Domains.}
    We mainly consider two cases:
    \begin{itemize}
      \item[$\bullet$] $\Omega=\mathbb{R}^d$: The whole Euclidean space;
      \item[$\bullet$] $\Omega=\mathcal{C}\subset\mathbb{R}^d$, a bounded convex set with $C^1$ boundary and nonempty interior.
      In this case, there exists $R_0>0$ such that $\mathcal{C}\subseteq \mathcal{B}(0;R_0)$.
    \end{itemize}
    \item[--] \textbf{Pushforward.} For $T:\mathbb{R}^d\to\mathbb{R}^d$, the pushforward $\mu=T_\#\nu$ is defined by:
    \begin{equation}
    \mu(A) = (T_\# \nu) (A) : = \nu(T^{-1}(A)) \quad\text{ for any}\quad A \subseteq Y\,.
    \end{equation} 
    If $T$ is a differentiable bijection and both $\mu$ and $\nu$ admit densities, then the change-of-variable formula reads
\begin{equation}\label{eqn:pushforward_density}
\mu(x) = \nu(T^{-1}(x))\big|\det J_{T^{-1}}(x)\big|.
\end{equation}
\end{itemize}

Since the regularity of $F$ is central to our discussion, it is important to clarify the relevant notions of differentiability. The first is the first variation.

\begin{definition}[First Variation (FV)]\label{def:first_var}
Let $F:\mathcal{P}_2(\Omega)\to\mathbb{R}\cup\{\infty\}$. A function $\left.\tfrac{\delta F}{\delta\rho}\right|_\rho:\Omega\to\mathbb{R}\cup\{\infty\}$ is the first variation of $F$ at $\rho$ if for all $\nu\in D(F)$:
\[
\int\left.\frac{\delta F}{\delta\rho}\right|_{\rho}(x)d(\nu-\rho)(x)=\left.\frac{d}{d\epsilon} F\left((1-\epsilon)\rho+\epsilon\nu\right)\right|_{\epsilon=0}\,.
\]
We denote by
\begin{equation}\label{eqn:first_var_set}
S_F^{\mathrm{f}} = \bigl\{\rho\in D(F):\ \nabla\bigl.\tfrac{\delta F}{\delta\rho}\bigr|_\rho \ \text{exists $\rho$-a.e.}\bigr\}
\end{equation}
the set of measures at which $F$ has a differentiable first variation.
\end{definition}

This definition effectively embeds $\mathcal{P}_2$ into a larger linear space of signed measures and performs differentiation in that ambient linear space.

The second notion, which underlies the gradient flow~\eqref{eqn:GF}, is based on the Wasserstein metric.
\begin{definition}[Wasserstein-2 Distance]\label{def:W2_subdiff}
For $\mu,\nu\in\mathcal{P}_2(\Omega)$, the Wasserstein-2 distance is
\begin{equation}\label{eqn:w_2}
W_2(\mu,\nu) = \inf_{\gamma\in\pi(\mu,\nu)}\sqrt{\int_{\Omega\times\Omega}|x-y|^2d\gamma(x,y)}\,,
\end{equation}
where $\Pi(\mu,\nu)$ is the set of couplings with marginals $\mu$ and $\nu$. The subset $\Pi_o(\mu,\nu)\subset\Pi(\mu,\nu)$ in which the infimum is attained is called the set of optimal transport plans.
\end{definition}

Under this metric, we can also define geodesic convexity and Wasserstein differentiability.

\begin{definition}[Geodesic Convexity]\label{def:W2_convexity}
Let $F:\mathcal{P}_2(\Omega)\to\mathbb{R}\cup\{\infty\}$ and $\mu_0,\mu_1\in\mathcal{P}_2(\Omega)$.
Let $\pi_o \in \Pi_o(\mu_0,\mu_1)$ be an optimal coupling and 
let $p^0, p^1 : \Omega\times\Omega \to \Omega$ be the projections onto the first and second coordinates.
Define the displacement interpolation $\mu_t := \big((1-t)p^0 + t p^1\big)_\# \pi_o$ for $t\in[0,1]$.
We say that $F$ is geodesically $m$-strongly convex if there exists some $m>0$ such that
\begin{equation}\label{eqn:W2_convexity}
F(\mu_t) \leq (1-t)F(\mu_0) + t F(\mu_1) - \tfrac{mt(1-t)}{2}  W_2^2(\mu_0,\mu_1),
\quad \forall t\in[0,1].
\end{equation}
If~\eqref{eqn:W2_convexity} holds with $m=0$, then $F$ is simply called geodesically convex.
\end{definition}

\begin{definition}[Wasserstein Differentiability]\label{def:W2_def}
Let $F:\mathcal{P}_2(\Omega)\to\mathbb{R}\cup\{\infty\}$ and $\rho\in D(F)$. A vector field $\xi\in L^2(\rho;\mathbb{R}^d)$ is a Wasserstein gradient of $F$ at $\rho$ if, for all $\nu\in D(F)$,
\[
F[\nu] = F[\rho] + \int_{\Omega\times\Omega}\langle\xi(x), y-x\rangle  d\pi(x,y) + o\big(W_2(\rho,\nu)\big)\,,
\]
for some $\pi\in\Pi_o(\rho,\nu)$. The collection of such gradients is denoted $\nabla_{W}F[\rho]$. We further set
\begin{equation}\label{eqn:W_2_diff_set}
S_F^{W} = \{\rho\in D(F): \nabla_{W}F[\rho]\neq\emptyset\}
\end{equation}
to be the collection of $\rho$ on which $F$ is Wasserstein-differentiable.
\end{definition}

When $F$ has specific structure, both the first variation and the Wasserstein gradient admit explicit formulas. A classical example is the Kullback-Leibler (KL) divergence against a target $\rho^*$:
\begin{equation}\label{eqn:relative_entropy}
F[\rho] = \KL[\rho|\rho^*] :=
\begin{cases}
\int_{\mathbb{R}^d}\tfrac{d\rho}{d\rho^*}\ln\left(\tfrac{d\rho}{d\rho^*}\right)d\rho^*\,, & \rho\ll \rho^*,\\
+\infty, & \text{otherwise}.
\end{cases}
\end{equation}
Thus $D(F)$ consists of measures absolutely continuous with respect to $\rho^*$. The following result summarizes differentiability properties~\cite[Theorem~10.4.9]{AGS08}:

\begin{theorem}\label{thm:subdifferentiability_entropy}
Let $F$ be defined as in~\eqref{eqn:relative_entropy}, and $\rho\in D(F)$. Write $\sigma=\frac{d\rho}{d\rho^*}$. Then:
\begin{itemize}
\item The first variation exists if and only if $\sigma>0$, $\rho^*$-a.e.;
\item $\rho$ is W-differentiable if and only if $\sigma\in W_{\mathrm{loc}}^{1,1}(\mathbb{R}^d)$ and $\nabla\ln\sigma\in L^2(\rho;\mathbb{R}^d)$.
\end{itemize}
If $\rho$ and $\rho^*$ admit densities $e^{-V}$ and $e^{-U}$, respectively, then
\begin{equation}\label{eqn:KL_FV}
\left.\tfrac{\delta F}{\delta\rho}\right|_\rho = 1+\ln\sigma = 1+U(x)-V(x)\,,
\end{equation}
and
\begin{equation}\label{eqn:KL_gradient}
\nabla_{W}F[\rho]
 = \nabla\ln\sigma = \nabla U(x)-\nabla V(x)\,.
\end{equation}
\end{theorem}

Running the gradient flow~\eqref{eqn:GF} for this $F$ yields an important property~\cite[Chapter 5, Section 2, Corollary 9]{K08}:

\begin{theorem}\label{thm:regularity_preservation}
Let $F$ be the KL functional against a target $\rho^* = e^{-U}$, with $U\in C^\infty(\Omega)$ and has Lipschitz gradient. Suppose the initial data $\rho_{\mathrm{ini}}\in S_F^{W}$, and let $\rho_t$ solve the gradient flow~\eqref{eqn:GF}. Then for all $t>0$, $\rho_t$ admits a density $\rho_t=e^{-V_t}$ with $V_t\in C^\infty_{\mathrm{loc}}(\Omega)$, and thus $\rho_t\in S_F^{W}$ for all $t>0$. Moreover, the gradient flow PDE takes the Fokker-Planck form
\begin{equation}
\partial_t\rho_t-\nabla\cdot(\rho_t\nabla U)=\Delta \rho_t\,.
\end{equation}
\end{theorem}

\begin{remark}\label{rmk:alarms}
Two observations are critical for later discussion:
\begin{itemize}
\item Theorem~\ref{thm:regularity_preservation} holds only in the continuous-time setting. The GF PDE preserves regularity, ensuring $\rho_t$ remains W-differentiable for all $t>0$. By contrast, the discrete FE scheme~\eqref{eqn:GF_dis} does not inherit this preservation (see Section~\ref{sec:counterexamples}).
\item Theorem~\ref{thm:subdifferentiability_entropy} shows that W-differentiability requires strong regularity conditions: if $F$ is W-differentiable at $\rho$, then not only does $\ln\sigma$ exists, is differentiable, and the derivative lies in $L^2(\rho;\mathbb{R}^d)$. Consequently,
\begin{equation}\label{eqn:set_relation}
S_F^{W} \subset S_F^{\mathrm{f}}\,,
\end{equation}
or equivalently,
\begin{equation}\label{eqn:set_relation2}
\text{first variation differentiable}\nRightarrow\text{W-differentiable}\,,
\end{equation}
meaning there exist $\rho\in S_F^{\mathrm{f}}\setminus S_F^{W}$ for which $F$ has a differentiable first variation, and $\nabla\frac{\delta F}{\delta\rho}$ appears formally computable, but is not W-differentiable.
\end{itemize}
\end{remark}

\subsection{Lions (L)-Differentiability}\label{sec:L-diff-def}

L-Differentiability~\cite{CD18} is another way to measure the regularity of a functional over probability measures. Instead of directly operate over the probability measure space, it associates each probability measure with a random variable and translates the regularity over to the study over the random variable. More specifically, for a functional $F$ that maps $\mathcal{P}_2(\Omega)$ to $\mathbb{R}$, we \emph{lift} it to be a functional of a random variable:
\begin{equation}\label{eqn:def_lift}
\tilde{F}(X) := F[\operatorname{Law}(X)]\,.
\end{equation}
Let $\nu$ be an atomless probability measure on a Polish space, and denote $z \sim \nu$ the underlying random variable, $X(z)$ can be seen as a map, and $\operatorname{Law}(X)=X_\#\nu$. When $\operatorname{Law}(X)\in\mathcal{P}_2(\Omega)$, we have:
\begin{itemize}
    \item[--] $X(z)\in\Omega$ for all $z$;
    \item[--] $\int |X(z)|^2d\nu(z)<\infty$.
\end{itemize}
So $X\in L^2(\nu; \Omega)$, a $\nu$-weighted $L_2$ function with values in $\Omega$. As a consequence, the lifting~\eqref{eqn:def_lift} naturally extends $F$, a functional over the probability measure space to $\tilde{F}$ defined over $L^2(\nu; \Omega)$.

One nice feature of this lifting is that now we operate on the $L_2$ function space, a linear function space, and the definition of the differentiability is much more standard in this space. This lifted version presents us an opportunity to define a new kind of differentiability that could potentially be easier to manipulate. Following~\cite[Definition 5.22]{CD18}:

\begin{definition}[L-Differentiability]\label{def:L-diff}
A function $F : \mathcal{P}_2(\Omega) \to \mathbb{R}$ is said to be \emph{L-differentiable} at $\rho \in \mathcal{P}_2(\Omega)$ if there exists a random variable $X$ with $\operatorname{Law}(X) = \rho$ such that the lifted function $\tilde{F}$ is Fr\'echet differentiable at $X$. The derivative is denoted as $D\tilde{F}(X)$, a $d$-dimensional vector valued function defined over $d\nu$.
\end{definition}

The definition, at its first sight, is a bit worrying, since the differentiability depends on the choice of the random variable $X$. This worry is removed by the following proposition.

\begin{proposition}\label{prop:L-diff-char}
Let $F$ be a real-valued function on $\mathcal{P}_2(\Omega)$, and let $\tilde{F}$ be its lifting to $L^2(\nu; \Omega)$. 
If $F$ is L-differentiable at $\rho \in \mathcal{P}_2(\Omega)$ in the sense of Definition~\ref{def:L-diff}, then:
\begin{enumerate}
    \item $\tilde{F}$ is Fr\'echet differentiable at every $X \in L^2(\nu; \Omega)$ with $\operatorname{Law}(X) = \rho$. Furthermore, $\operatorname{Law}(X, D\tilde{F}(X))$ is independent of the choice of such $X$.
    \item There exists a measurable function $\xi(x) : \Omega \to \mathbb{R}^d$ such that
    \[
    D\tilde{F}(X)(z) = \xi(X(z)) \quad \text{a.s.}
    \]
    for all $X \in L^2(\nu; \Omega)$ whose $\operatorname{Law}(X) = \rho$.
\end{enumerate}
\end{proposition}

This proposition is adopted from~\cite[Proposition 5.24 and 5.25]{CD18} for $\Omega=\mathbb{R}^d$. 
In the case when $\Omega = \mathcal{C}$, the proof only needs to go through slight changes. 
Namely, replacing the countable covering by a finite covering, and employing a mollifier adapted to $\mathcal{C}$ instead of the standard Euclidean mollifier.

The significance of the proposition is two-folded. First, by removing the dependence on a particular choice of $X$, it guarantees that L‑derivative is a well-defined quantity. Furthermore, the proposition also suggests the derivative $D\tilde{F}$ is in a form of composition. Denoting the L-derivative over $\mathcal{P}_2$ and functional derivative over $L_2(\nu;\Omega)$:
\begin{equation}\label{eqn:L-grad-def}
    \partial_\rho F[\rho](x)=\xi(x)\,,\quad\text{and}\quad G[X](z) = D\tilde{F}(X)(z)=\xi(X(z))\,,
\end{equation}
we have the equivalence:
\begin{equation}\label{eqn:L-grad-def-equal}
G[X](z) = \partial_\rho F[X_\# \nu]\circ X\,. 
\end{equation}
This nice link between the functional gradient over $L^2(\nu;\Omega)$ and the L-derivative over $\mathcal{P}_2$ will be levied significantly in Section~\ref{sec:regularized_KL}.
\subsection{Comparison between three kinds of differentiability}\label{sec:induction_with_L}
So far, we have introduced three notions of differentiability for functionals on $\mathcal{P}_2$:
\begin{itemize}
    \item[$\bullet$] First variation (FV), obtained by embedding $\mathcal{P}_2$ into the linear space of signed measures;
    \item[$\bullet$] W-differentiability, arising from the geometric structure induced by the $W_2$ metric;
    \item[$\bullet$] L-differentiability, defined by lifting probability measures to square-integrable random variables, thus a weighted $L_2$ space.
\end{itemize}

These notions are not completely decoupled. In fact, when the functional is sufficiently smooth in its first variation, L-differentiability and W-differentiability can be deduced. We make this precise below.

\begin{proposition}\label{prop:L-diff-bootstrap}
Let $F:\mathcal{P}_2(\Omega)\to\mathbb{R}$ be a functional with first variation $\tfrac{\delta F}{\delta\rho}$. Assume $S_F^{\mathrm{f}}=\mathcal{P}_2(\Omega)$ and suppose:
\begin{enumerate}
\item $\tfrac{\delta F}{\delta\rho}$ and $\nabla\tfrac{\delta F}{\delta\rho}$ are continuous in the product topology $\mathcal{P}_2(\Omega)\times\Omega$;
\item For any bounded subset $\mathcal{K}\subseteq\mathcal{P}_2(\Omega)$ (bounded in $W_2$), the function $\left.\tfrac{\delta F}{\delta\rho}\right|_\rho(x)$ has at most quadratic growth in $x$, uniformly for $\rho\in\mathcal{K}$, and $\left.\nabla\tfrac{\delta F}{\delta\rho}\right|_\rho(x)$ has at most linear growth in $x$, uniformly for $\rho\in\mathcal{K}$.
\end{enumerate}
Then $F$ is continuously L-differentiable on $\mathcal{P}_2(\Omega)$, with continuity understood in the product topology $\mathcal{P}_2(\Omega)\times\Omega$. Moreover, the L-derivative is explicit:
\[
\partial_\rho F[\rho](x) = \left.\nabla\tfrac{\delta F}{\delta\rho}\right|_{\rho}(x) , \quad \text{for all } \rho \in \mathcal{P}_2(\Omega), x \in \Omega \,.
\]
\end{proposition}

Thus, though first variation does not deduce L-differentiability in general, it does deduce continuous L-differentiability when first variation and its gradient are sufficiently smooth and their growth are controlled.

Furthermore, continuous L-differentiability can deduce W-differentiability.

\begin{proposition}\label{prop:L-diff_implies_W-diff}
Let $F:\mathcal{P}_2(\Omega)\to\mathbb{R}$ be continuously L-differentiable. Then $F$ is also Wasserstein differentiable in the sense of Definition~\ref{def:W2_subdiff}, with
\[
\nabla_{W}F[\rho](x) = \partial_\rho F[\rho](x) ,
\quad \text{for all } \rho\in\mathcal{P}_2(\Omega),\, x\in\Omega .
\]
\end{proposition}

These results were adopted from~\cite[Proposition 5.48 and Theorem 5.64]{CD18} where the authors considered $\Omega=\mathbb{R}^d$. 
The results are kept the same with a small modification to the proof when $\Omega=\mathcal{C}$. Together, Propositions~\ref{prop:L-diff-bootstrap} and~\ref{prop:L-diff_implies_W-diff} establish a hierarchy:
\begin{equation}\label{eqn:hierarchy}
    \text{smooth FV}\quad\Rightarrow\quad \text{continuous L-differentiability}\quad\Rightarrow\quad\text{W-differentiability}
\end{equation}

This hierarchy clarifies the connections among the three notions, and will be explicitly used for justifying regularity in Section~\ref{sec:regularized_KL}.

\section{Breakdowns and counter‑examples}\label{sec:counterexamples}
The discussion in Section~\ref{sec:prelim} highlights potential pitfalls of applying the FE scheme to gradient flows, especially comparing~\eqref{eqn:set_relation2} and~\eqref{eqn:hierarchy}. Consider the following simple setup: let the target distribution be $\rho^*=e^{-U}$ for a smooth, symmetric potential $U$ over $\mathbb{R}$, and choose the initial distribution
\begin{equation}\label{eqn:piecewise}
\rho_0(x)=\rho_{\mathrm{ini}}(x) := \begin{cases}
a \cdot e^{-U(x)} & x \in [0,1) \\
b \cdot e^{-U(x)} & x \in (1,+\infty) \\
\rho_0(-x) & x \in (-\infty, 0)
\end{cases}\,,
\end{equation}
with $a,b\neq 1$ but adjusted for normalization. Then:
\begin{itemize}
    \item[1.] The discontinuity of the density at $x=1$ implies $\rho_0\notin S_F^{W}$, by Theorem~\ref{thm:subdifferentiability_entropy}. Hence the gradient flow~\eqref{eqn:GF} is not well defined at $t=0$;
    \item[2.] The equivalent PDE form~\eqref{eqn:FP}, however, is of diffusion type. It instantaneously smooth the solution, so that $\rho_t\in C^\infty$ for all $t>0$.
    \item[3.] Even though $\rho_0\notin S_F^{W}$, the expression $\nabla\frac{\delta F}{\delta\rho}$ remains computable: $\nabla\frac{\delta F}{\delta\rho}=\nabla(\ln\rho_{0}-U)=0$ $a.e.$. Consequently, the FE step yields the trivial update $\rho_n=\rho_0$ for all $n$.
\end{itemize}
The comparison between (2) and (3) illustrates the discrepancy: the discrete iterates $\rho_n$ do not match the continuous solution $\rho_{t=nh}$, no matter how small $h$ is. This shows the big contrast between interpreting the gradient flow as a convection-diffusion PDE and as a transport-type equation.

This toy example is somewhat artificial, since the initial data is already outside $S_F^{W}$. Nevertheless, it reveals the crux of the issue. We will argue below that even if $\rho_0$ lies within $S_F^{W}$, so that the GF equation is perfectly well defined, the FE scheme alone can still drive the solution out of $S_F^{W}$ and cause failure. In other words, the breakdown originates from the discretization, not from the PDE.

The mechanism mirrors Remark~\ref{rmk:alarms}. 
Theorem~\ref{thm:regularity_preservation} guarantees that W-differentiability is preserved in continuous time. 
By contrast, FE does not preserve this property. For any given $\rho^*$, one can construct an initial $\rho_{\mathrm{ini}}$ with just enough regularity to begin the evolution, lying on the boundary of $S_F^{W}$. 
The velocity field~\eqref{eqn:KL_gradient} then immediately pushes the iterate outside $S_F^{W}$, in arbitrarily small time. 
At that point, the GF structure breaks down.

The situation is made worse by the fact that, as Theorem~\ref{thm:subdifferentiability_entropy} shows, the formal expression $\nabla \tfrac{\delta F}{\delta\rho}$ may exist even when $F$ is not W-differentiable at $\rho$. Thus, one may continue to compute updates as if the gradient were valid, with no warning that the underlying structure has collapsed. This gives the false impression of progress, while concealing the fundamental failure of the scheme.

Taken together, these observations show that even with a smooth log-concave target and arbitrarily small step sizes, FE can fail dramatically: the discrete solution may deviate from the true gradient flow by an $\mathcal{O}(1)$ error. Explicit counterexamples are provided in Section~\ref{sec:examples}.

It should be stressed that this loss of regularity is a generic phenomenon. Our construction is extreme in that the discrepancy appears at the very first step. In general, unless the initial distribution is $C^\infty$, the solver will fail after finitely many steps. Section~\ref{sec:regularity} analyzes this general mechanism in detail.

We will end this discussion with a comparison with~\cite{CSS23P}, \cite{CSS25} in which the authors proved the convergence of FE when the $\lambda$-dissipative condition is satisfied. The counter-examples that we show below do not satisfy the $\lambda$-dissipative condition. Indeed, using~\eqref{eqn:piecewise} as the example, since the formal gradient $\nabla \frac{\delta F}{\delta \rho}=0$, we cannot find any $\lambda<0$ to allow the $\lambda$-dissipativity condition, even though $\KL[\rho|\rho^*]$ remains finite.

\subsection{Examples}\label{sec:examples}
We now present two examples in which the initial data lies inside $S_F^{W}$. By Theorem~\ref{thm:regularity_preservation}, the continuous gradient flow preserves $\rho_t\in S_F^{W}$ for all $t>0$. In contrast, we will show that the discrete FE scheme~\eqref{eqn:GF_dis} produces an iterate $\rho_1$ that leaves $S_F^{W}$ after just a single step while remaining within $S_F^{\mathrm{f}}$.

\subsubsection{Example 1: Loss of regularity due to the noninjectivity of the pushforward map}\label{sec:example_1}
Consider the target distribution
\begin{equation}\label{eqn:rho__ex1}
\rho^* = e^{-U}, \qquad U(x)=\tfrac{x^2}{2}+\tfrac{x^4}{4}+C_0,
\end{equation}
with $C_0$ chosen so that $\rho^*$ is normalized: $C_0 = \ln\left(\int_{\mathbb{R}} \exp(-\frac{x^2}{2} - \frac{x^4}{4})\rd x\right)$. Since $U$ is convex, $F$ is convex over $\mathcal{P}_2^r(\mathbb{R})$, with minimum $\min_{\rho} F=0$ attained at $\rho_{\mathrm{opt}}=\rho^*$. The initial distribution is taken as the standard Gaussian,
\begin{equation}\label{eqn:rho_0_ex1}
\rho_0=\rho_{\mathrm{ini}}=e^{-V_0},
\qquad V_0(x)=\tfrac{x^2}{2}+\ln(2\sqrt{\pi}).
\end{equation}
We show that after one FE step of any size $h>0$, the iterate $\rho_1\notin S_F^{W}$.

\begin{proposition}
With the setting above, the first FE iterate $\rho_1$ is Lebesgue absolutely continuous and that $\rho_1\in S_F^{\mathrm{f}}$. However, its density is discontinuous and hence $\rho_1 \not \in S_F^{W}$.
Both claims are independent of $h$.
\end{proposition}
\begin{proof}
According to~\eqref{eqn:GF_dis},
\[
\rho_1 = T_\# \rho_0 \quad\text{ with }\quad  T(x) := x - h \left.\nabla \frac{\delta F}{\delta \rho}\right|_{\rho_0}(x) \ .
\]
Using~\eqref{eqn:KL_gradient}, we compute
\begin{equation}\label{eqn:T_ex1}
T(x)=x-h(\nabla U(x)-\nabla V(x))=x-hx^3.
\end{equation}
The map $T$ is not injective. However, it becomes injective when restricted to three disjoint intervals. Define $r=\tfrac{2}{3}\sqrt{1/(3h)}$. Then $T$ decomposes as in Table~\ref{tab:my_label}, where each restriction $T_i$ is one-to-one.
\begin{table}[ht!]
    \centering
    \begin{tabular}{|c|c|c|}
    \hline
    Map & domain of $x$ & range of $y$ \\
    \hline
         $T_1$ & $(-\infty, -\frac{3}{2}r)$ & $y \in (-r,+\infty)$\\
         $T_2$ & $(-\frac{3}{2}r,\frac{3}{2}r)$ & $y \in (-r,r)$ \\
         $T_3$ & $(\frac{3}{2}r,+\infty)$ & $y \in (-\infty, r)$ \\
         \hline
    \end{tabular}    
    \caption{Ranges of $T$ on domains of injectivity.}\label{tab:my_label}
\end{table}

To compute $\rho_1$, one must account for all preimages of each $y$. Because $T$ is odd, it suffices to consider $y>0$.
\begin{itemize}
\item For $y\in(r,\infty)$, there is a unique preimage from $T_1^{-1}$. Using the pushforward formula~\eqref{eqn:pushforward_density},
\begin{equation}\label{eqn:p1_ex1_a}
\rho_1(y) = \rho_0(T_1^{-1}(y)) \vert \J_1(y) \vert \,.
\end{equation}

\item For \(y=r\) there are exactly two preimages: one obtained from \(T_{1}\), and the other is \(\tfrac{3}{2}r\).  
Since \(\rho_{0}\) is Lebesgue‑absolutely continuous, it assigns zero mass to any single point.  Consequently,
\[
\rho_{1}(\{r\})=\rho_{0}(T_{1}^{-1}(\{r\}))+\rho_{0}(\{\frac{3}{2}r\})=0 .
\]

\item For $y\in(0,r)$, there are three preimages from $T_1^{-1},T_2^{-1},T_3^{-1}$. The density is therefore:
\begin{equation}\label{eqn:p1_ex1_b}
\begin{aligned}
& \rho_1(y) =& \rho_0(T_1^{-1}(y)) \vert  \J_1(y) \vert + \rho_0(T_2^{-1}(y)) \vert  \J_2(y) \vert + \rho_0(T_3^{-1}(y)) \vert \J_3(y) \vert  \,.
\end{aligned}
\end{equation}
\end{itemize}
Here $|J_i(y)|=|\det J_{T_i^{-1}}(y)|$. Within each domain $T_i(x)=x-hx^3$, so $\det J_T(x)=1-3hx^2$, and thus $|J_i(y)|=1/|1-3hx^2|$ with $x=T_i^{-1}(y)$.

The polynomial form of $T$ and the inverse function theorem guarantee that $\rho_{1}$ is $C^\infty$ on the intervals $(0,r)$ and $(r,+\infty)$. Because $\rho_{1}(\{r\})=0$, $\rho_{1}$ is absolutely continuous with respect to Lebesgue. Equation~\eqref{eqn:KL_gradient} then shows that $\left.\nabla \frac{\delta F}{\delta\rho}\right|_{\rho_1}$
is defined $\rho_{1}$-almost everywhere, and consequently $\rho_{1}\in S_{F}^{\mathrm f}$.

We claim that
\begin{equation}\label{eqn:jump}
\lim_{y\uparrow r}\rho_1(y)=\infty\,,
\qquad\text{while}\qquad
\lim_{y\downarrow r}\rho_1(y)<\infty\,,
\end{equation}
showing a jump discontinuity at $y=r$.

Indeed, from~\eqref{eqn:p1_ex1_b}, as $y\uparrow r$ one has $T_2^{-1}(y)\to \tfrac{3}{2}r$, so $\rho_0(T_2^{-1}(y))\to\rho_0(\tfrac{3}{2}r)>0$. Meanwhile
\begin{equation}\label{eqn:limit_J_singular}
\lim_{y\uparrow r}|J_2(y)|=\lim_{x\uparrow \tfrac{3}{2}r}\tfrac{1}{|1-3hx^2|}=+\infty.
\end{equation}
A similar argument applies to the $T_3$ contribution. This proves~\eqref{eqn:jump} and establishes the discontinuity, hence $\rho_1\notin S_F^{W}$.
\end{proof}

\begin{figure}[htb]
    \centering
    \includegraphics[width = 0.48\textwidth]{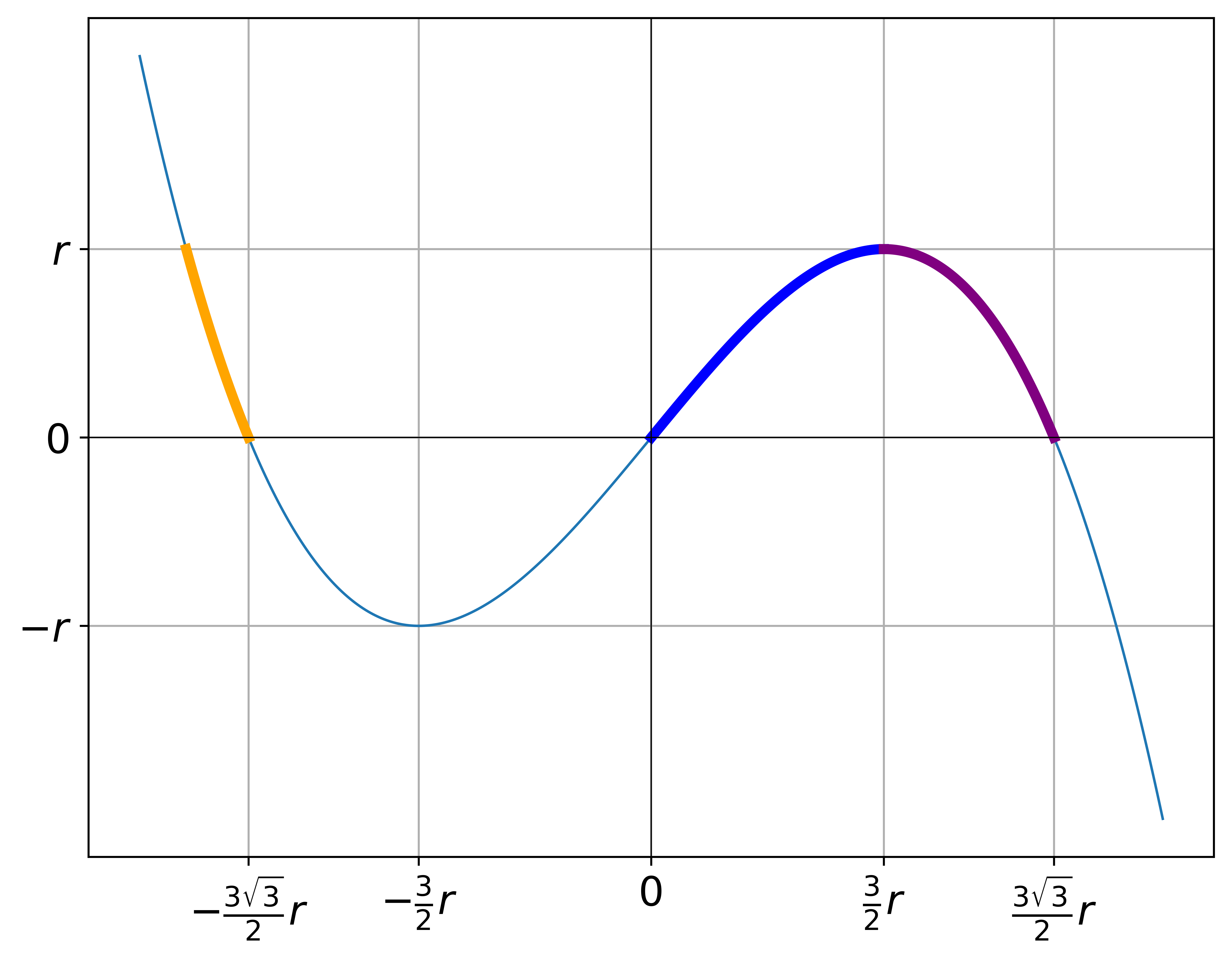}
    \includegraphics[width = 0.48\textwidth]{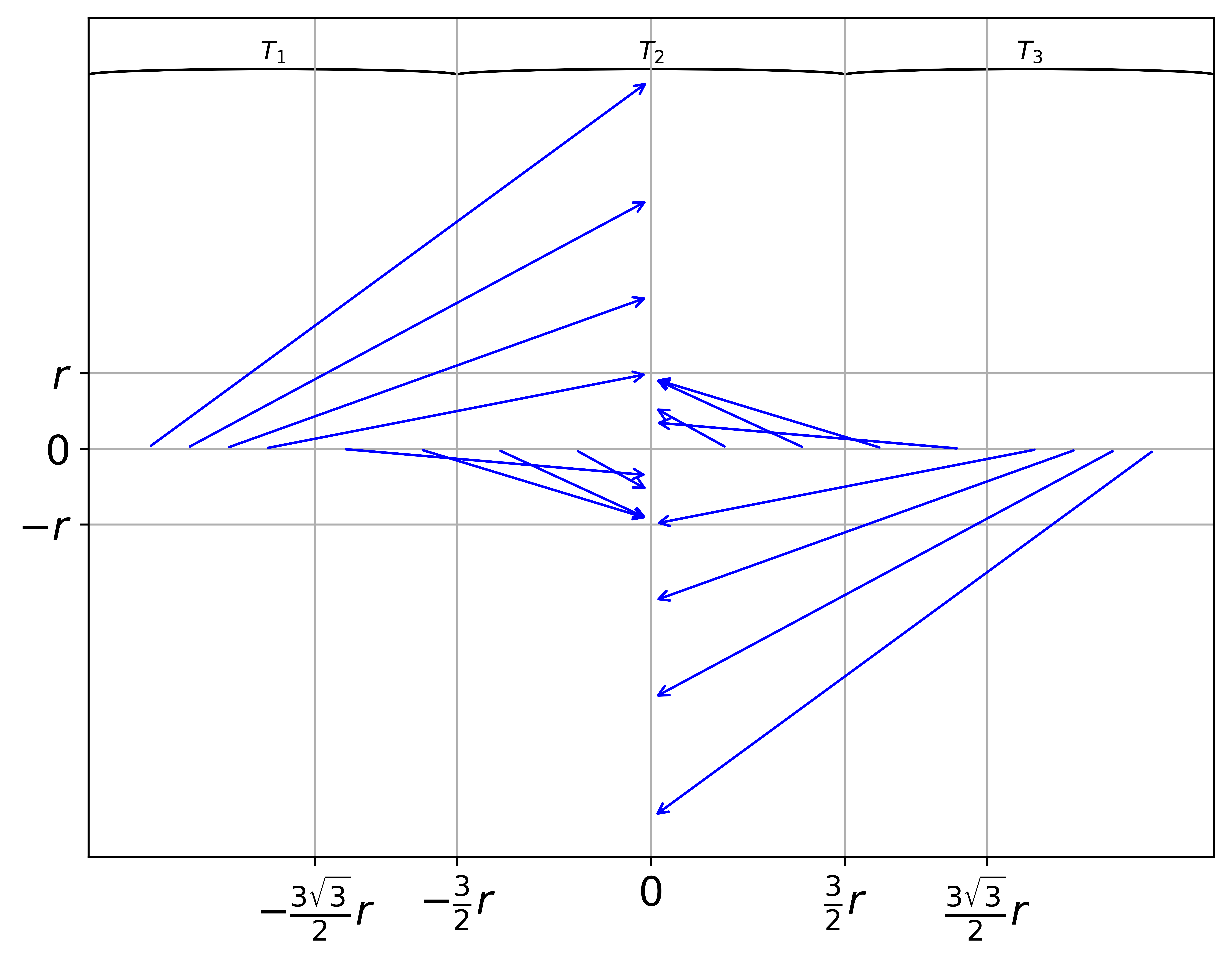}
    \caption{Left: the pushforward map $T$. For $y\in(0,r)$, three preimages exist, marked by three bold-lines. Injectivity is resumed when $T$ is restricted to three non-intersecting domains. Right: a decomposition of \(T\) into its components \(T_1,T_2,T_3\). The blue arrows illustrate how points on the horizontal \(x\)-axis are mapped onto the vertical \(y\)-axis, and the labels at the top identify the domain of each $T_i$-s.}
\label{fig:pushforward_T}
\end{figure}

Figure~\ref{fig:pushforward_T} illustrates the structure of $T$. Despite both the target and initial distributions being smooth and log-concave, the FE update introduces a discontinuity. The problem stems entirely from the non-injectivity of $T$: for $y>r$, only $T_1$ contributes, while for $y\in(0,r)$ all three branches contribute. Because the $T_2$ and $T_3$ contributions remain nonzero at $y=r$, a jump discontinuity necessarily arises, for any finite $h$.

\subsubsection{Example 2: Loss of regularity due to the consumption of derivatives}\label{sec:example_2}
For the second example, consider the target distribution
\begin{equation}\label{eqn:rho__ex2}
\rho^* = e^{-U}, \qquad U(x)=\tfrac{x^2}{2}+\ln(2\sqrt{\pi}),
\end{equation}
where the constant ensures normalization. In this case, $F$ is $1$-convex on $\mathcal{P}_2^r(\mathbb{R})$, with unique minimizer $\rho^*$. The initial distribution is chosen as
\begin{equation}\label{eqn:def_p0_ex2}
\rho_0(x)=\rho_{\mathrm{ini}}=\frac{1}{D_0} \exp(-V_0(x))\,,\quad \text{with}\quad V_0(x) := \begin{cases}
\frac{x^2}{2} & x \in (-1,1) \\
\vert x \vert - \frac{1}{2} & \text{Otherwise}
\end{cases} \,,
\end{equation}
with normalization constant
\begin{equation}\label{eqn:def_d0_reg}
D_0=\int_{\mathbb{R}} e^{-V_0(x)}dx
= \sqrt{2\pi}\mathrm{Erf}\Bigl(\tfrac{1}{\sqrt{2}}\Bigr)+2e^{-1/2}\,.
\end{equation}
The FE scheme~\eqref{eqn:GF_dis} updates $\rho_{n+1}=(T_n)_\#\rho_n$, where $T_n(x) := x - h_n \nabla\left. \frac{\delta F}{\delta \rho}\right|_{\rho_n}(x)$ and $h_n \in (0,1)$ is the step-size at $n$-th iteration. In this setting, the following proposition holds true:
\begin{proposition}\label{prop:ex2_property}
With the setting above, 
\begin{enumerate}
  \item The measure $\rho_{1}$ is Lebesgue absolutely continuous but its density has measure-zero discontinuity, making $\rho_1 \in S_F^{\mathrm{f}}\backslash S_F^{W}$.
  \item For each $n \geq 0$ there exist constants
        $a_{n},c_{n}\in[1,\infty)$ and $b_{n}\in\mathbb R$ such that
\begin{equation}\label{eqn:def_pk_ex2}
\rho_n(x) = \begin{cases}
\frac{1}{D_0}\exp(-\frac{x^2}{2}) & x \in [0,1) \\
0 & x \in [1,c_n) \\
\exp(-a_n x + b_n) & x \in [c_n, +\infty)\\
\rho_n(-x) & x \in (-\infty,0)
\end{cases}\,.
\end{equation}
\end{enumerate}
\end{proposition}

The proof is direct computation, and is deferred to Appendix~\ref{sec:Appendix_proof_prop:ex2_property}. This explicit formula~\eqref{eqn:def_pk_ex2} allows us to obtain explicit estimate.

\begin{theorem}\label{thm:ex2_lowerbound}
Under the conditions of Proposition~\ref{prop:ex2_property}, we have
\[
F[\rho_{n}]>0.019\,,\quad \forall n\geq 0\,.
\]
\end{theorem}

\begin{proof}
By the Pinsker's inequality:
\[
\begin{aligned}
F[\rho_n] - F_* & = \KL[\rho_n \vert e^{-U}] \geq 2 \left(\text{TV}(\rho_n, e^{-U})\right)^2 \\
& \geq 2 \left(\int_{-1}^{1} (\rho_n(x) - \rho^*(x)) \rd x\right)^2 \\
& = 2 \left((\frac{1}{\sqrt{2\pi}} - \frac{1}{D_0}) \int_{-1}^{1} \exp(-\frac{x^2}{2}) \rd x\right)^2 \\
& = 4 \pi \left(\text{Erf}(\frac{1}{\sqrt{2}})\right)^2 \left(\frac{1}{\sqrt{2 \pi}} - \frac{1}{\sqrt{2 \pi}\text{Erf}(\frac{1}{\sqrt{2}}) + \frac{2}{\sqrt{e}}}\right)^2 > 0.019 \,.
\end{aligned}
\]
Here the second inequality comes from lower bounding the total variation using the total variation on the subset $(-1,1)$; the third line comes from plugging in~\eqref{eqn:def_pk_ex2}; the last line comes from using the definition of the normalization constant $D_0$ in Equation~\eqref{eqn:def_d0_reg}; and the last inequality comes from a straightforward numerical evaluation.
\end{proof}

This example illustrates the critical role of regularity. Even though the target is Gaussian and the initial distribution is well behaved, the FE scheme destroys regularity after just one step, and the solver provides no indication of this breakdown. The density develops discontinuities, yet the scheme continues producing iterates that remain at a fixed nontrivial distance from the minimizer. In particular, $F[\rho_n]$ never approaches zero, showing that FE cannot approximate the true gradient flow, regardless of the step size.

\begin{figure}[htb]
   \centering
   \includegraphics[width = 0.8\textwidth]{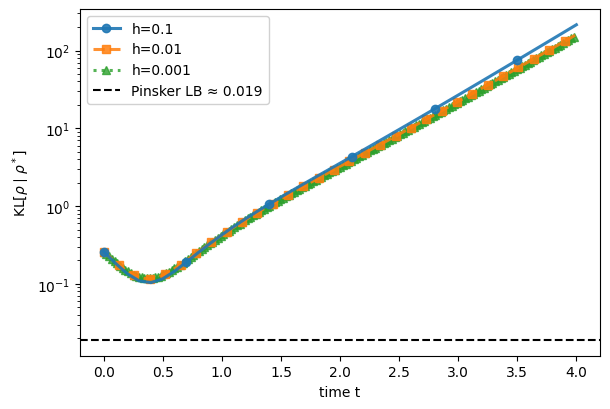}
   \caption{$\KL[\rho_n \vert \rho^*]$ as a function of time, plotted in semi-log scale.
   The blue (circle-marked), orange (square-marked), and green (triangle-marked) curves correspond to step sizes $h=0.1$, $h=0.01$, and $h=0.001$, respectively.
   The dashed black line represents the theoretical lower bound $0.019$ established in Theorem~\ref{thm:ex2_lowerbound}.}
   \label{fig:KL_counterexample}
\end{figure}

To further illustrate this phenomenon, we conduct a numerical experiment. 
For this example, one can derive explicit formula for the coefficients $(a_n,b_n,c_n)$ in Proposition~\ref{prop:ex2_property} (see Appendix~\ref{sec:Appendix_proof_prop:ex2_property}), and compute $\KL[\rho_n \vert \rho^*]$ using the explicit formula in Lemma~\ref{lem:KL_formula_ex2}.
We run forward Euler (FE) iteration up to $T=4$ with several different step sizes. 
As shown in Figure~\ref{fig:KL_counterexample}, the Kullback--Leibler divergence $\KL[\rho_n\vert\rho^*]$ decreases slightly at the beginning, but soon grows rapidly and remains far above the analytical lower bound $0.019$ for all step sizes. Different curves correspond to different step sizes $h$ and they nearly coincide when $h$ becomes small, confirming that this failure is intrinsic to the FE discretization, agreeing with the conclusion from Proposition~\ref{prop:ex2_property}.

\subsection{The loss of regularity is generic}\label{sec:regularity}
The examples above show that the regularity of $\rho_n$ deteriorates when propagated by the FE scheme~\eqref{eqn:GF_dis}. This phenomenon is not isolated: it occurs generically for any FE solver applied to gradient flows driven by KL-type functionals. We formalize this below.

\begin{proposition}
Let $F[\rho] = \KL[\rho \vert e^{-U}]$ with $U \in C^\infty(\mathbb{R}^d)$. Assume $U$ is gradient Lipschitz, meaning $\Hess[U] \preceq M I$ for some $M \in (0,+\infty)$. Let $\rho_0=e^{-V_0}$ with $V_0\in C^{m+2}(\mathbb{R}^d)$ smooth and $\Hess[V_0]\succeq -M_0 I$ for some $M_0>0$. Then the one-step FE update
\[
\rho_1 = (T_0)_\#\rho_0=e^{-V_1}\,,\quad\text{with}\quad T_0(x) = x - h_0 \left.\nabla \frac{\delta F}{\delta \rho}\right|_{\rho_0}(x)
\]
necessarily has reduced regularity. Specifically, for $h_0\in(0,\tfrac{1}{M+M_0})$, we have $V_1\in C^m$, i.e.\ $V_1$ has two fewer derivatives than $V_0$.
\end{proposition}

\begin{proof}
From~\eqref{eqn:GF_dis} and~\eqref{eqn:KL_gradient},
\[
T_0(x) = x - h_0 \nabla \left.\frac{\delta F}{\delta \rho}\right|_{\rho_0}(x) = x + h_0 \nabla V_0(x) - h_0 \nabla U(x) \,.
\]
Since $V_0 \in C^{m+2}$, $T_0 \in C^{m+1}$ and its Jacobian is
\[
\J_{T_0}(x) = I + h_0 \cdot \Hess[V_0](x) - h_0 \cdot \Hess[U](x) \succ 0 \,.
\]
The positivity of the Jacobian comes from the choice of $h_0$. It also implies that $T_0$ is injective, and thus by the Global Inverse Function Theorem, $T_0^{-1}$ is well-defined and $T_0^{-1} \in C^{m+1}$. By~\eqref{eqn:pushforward_density},
\[
\begin{aligned}
\exp\left(-V_1\left(x\right)\right) & = \left(\exp\left(-V_0\left(T_0^{-1}\left(x\right)\right)\right)\right) \left\vert \J_{T_0^{-1}}\left(x\right) \right\vert \\
& = \exp\left(-\left(V_0\left(T_0^{-1}\left(x\right)\right) - \ln\left(\left\vert \J_{T_0^{-1}}\left(x\right) \right\vert\right)\right)\right) \,,
\end{aligned}
\]
namely:
\[
V_1\left(x\right) = V_0\left(T_0^{-1}\left(x\right)\right) - \ln\left(\left\vert \J_{T_0^{-1}}\left(x\right) \right\vert\right) \,.
\]
Recall that $T_0^{-1} \in C^{m+1}$ and $V_0 \in C^{m+2}$, and the Jacobian is taking one higher order of derivative on $V_0$, so $\ln\left(\left\vert \J_{T_0^{-1}}\left(x\right) \right\vert\right) \in C^{m+1}$, and we conclude that $V_1 \in C^{m}$.
\end{proof}

The proposition shows that each application of \eqref{eqn:GF_dis}
loses two derivatives. Even assuming that all $T_k$ are invertible, any initial density with only a finite
number of derivatives will eventually lose all regularity and enter the
regime where the Wasserstein gradient is not defined.  This highlights that
the simple Forward Euler stepping should be used with caution for
gradient flows.

\section{Regularized Kullback-Leibler Divergence}\label{sec:regularized_KL}
The break-down of the simplest FE scheme applied on GF is disappointing. Yet, the analysis shed the light on the reason of the breakdown: the loss of regularity. KL divergence against a target distribution $\rho^*$ nicely behaves when it is confined on the set that is absolutely continuous against $\rho^*$, but the functional becomes $\infty$ outside this set. However, the underlying Wasserstein metric does not sense this difference. For any $\rho$ that is absolutely continuous against $\rho^*$, one can find a small perturbation to it that drives $F$ to infinity, making derivative uncomputable.

To address this issue, clearly, we need to modify the functional and replace it with a nicer behaved one. A natural candidate is the Gaussian-regularized Kullback--Leibler divergence (also known as the regularized internal energy or
blob method~\cite{Carrillo2019}, \cite{CJT25}), defined by
\begin{equation}\label{eqn:def_reg_KL}
F^\epsilon[\rho] := \KL^\epsilon[\rho|\rho^*]=\int_{\mathcal{C}} \Big( U(x) + \ln\!\big((\varphi_\epsilon * \rho)(x)\big)\Big)\, d\rho(x),
\end{equation}
where $\rho^*=e^{-U}$ and $\varphi_\epsilon(x) := \exp\!\big(-\tfrac{\|x\|_2^2}{2\epsilon}\big)$ is the Gaussian kernel and 
$\epsilon>0$ denotes the regularization parameter. Throughout the section, we confine ourselves to the setting where $\Omega = \mathcal{C}$.  For brevity, we will refer to $F^\epsilon$ simply as the \emph{regularized KL}. 

\subsection{Smoothness of the Regularized KL Divergence}
We are to show when the target distribution $\rho^*$ is smooth, this regularized KL divergence posses nice regularities, and has well-defined L‑derivative and Wasserstein gradient, both precisely given by $\left.\nabla\tfrac{\delta F}{\delta\rho}\right|_{\rho}$. The assumption on $U$ is:

\begin{assumption}[Potential function]\label{asp:potential_function}
\begin{itemize}
\item $\nabla U$ is $C_1$-Lipschitz on $\mathcal{C}$ for some $C_1>0$, i.e. 
\[
\|\nabla U(x)-\nabla U(y)\|_2 \leq C_1 \|x-y\|_2, \quad \forall x, y \in \mathcal{C}.
\]
\item At some reference point $x_0\in \mathcal{C}$ we have $|U(x_0)| \leq C_0$ and $\|\nabla U(x_0)\|_2 \leq C_2$ for constants $C_0,C_2>0$.
\end{itemize}
\end{assumption}
We should note that in~\cite[Propositions 3.10 and 3.12]{Carrillo2019} the authors have proved Wasserstein subdifferentiability for regularized R\'enyi divergences with exponent $m\geq 2$, but their results do not apply to the entropy case. Under this assumption, we can establish the regularity of the first variation.

\begin{lemma}\label{lem:reg_reg_KL_first_var}
Under Assumption~\ref{asp:potential_function}, $\tfrac{\delta F^\epsilon}{\delta\rho}$ and $\nabla\tfrac{\delta F^\epsilon}{\delta\rho}$ are uniformly bounded and continuous on the product topology $\mathcal{P}_2(\mathcal{C}) \times \mathcal{C}$. Namely, there are constants $0<L_i<\infty$ depending only on $C_0$, $C_1$, $C_2$, $R_0$ and $\epsilon$, so that for any $\rho_1, \rho_2 \in \mathcal{P}_2(\mathcal{C})$ and any $x,y \in \mathcal{C}$, 
\begin{equation}\label{eqn:reg_KL_first_var_lip}
\left\vert \left.\tfrac{\delta F^\epsilon}{\delta\rho}\right|_{\rho_1}(x) - \left.\tfrac{\delta F^\epsilon}{\delta\rho}\right|_{\rho_2}(y) \right\vert \leq L_1  W_2(\rho_1, \rho_2) + L_2 \left\Vert x - y \right\Vert_2 \, ,
\end{equation} and
\begin{equation}\label{eqn:reg_KL_grad_first_var_lip}
\left\Vert \left.\nabla\tfrac{\delta F^\epsilon}{\delta\rho}\right|_{\rho_1}(x) - \left.\nabla\tfrac{\delta F^\epsilon}{\delta\rho}\right|_{\rho_2}(y) \right\Vert_2 \leq L_3  W_2(\rho_1, \rho_2) + L_4 \left\Vert x - y \right\Vert_2 \,.
\end{equation}
Furthermore, for any $\pi \in \Pi(\rho_1, \rho_2)$, 
\begin{equation}\label{eqn:reg_KL_grad_transport_Lip}
{\int \left\Vert \left.\nabla\tfrac{\delta F^\epsilon}{\delta\rho}\right|_{\rho_1}(x) - \left.\nabla\tfrac{\delta F^\epsilon}{\delta\rho}\right|_{\rho_2}(y) \right\Vert_2^2 d \pi(x,y) } \leq \left(C_1 + \frac{3}{\epsilon} e^{\frac{8 R_0^2}{\epsilon}}\right)^2{{\int \left\Vert x - y \right\Vert_2^2 d \pi(x,y)}} \, .\end{equation}

\end{lemma}

The proof of the above lemma is deferred to Appendix~\ref{sec:Supplement_proof_regularized_KL}. 
It is essentially brute-force calculation that requires explicit computation of the first variation and its derivative.

According to the discussion in Section~\ref{sec:induction_with_L}, this regularity leads to the L- and W-differentiability.

\begin{theorem}\label{thm:reg_KL_grad_coincide}
Under Assumption~\ref{asp:potential_function}, $F^\epsilon$ is both L- and W-differentiable on $\mathcal{P}_2(\mathcal{C})$, and the two gradients coincide:
\[
\nabla_{W}F^\epsilon[\rho] = \partial_\rho F^\epsilon [\rho] = \left.\nabla\tfrac{\delta F^\epsilon}{\delta\rho}\right|_{\rho} \,, \quad \forall \rho \in \mathcal{P}_2(\mathcal{C}) \,.
\]
\end{theorem}
This theorem comes from straightforward application of Proposition~\ref{prop:L-diff-bootstrap} and~\ref{prop:L-diff_implies_W-diff}.
As a consequence, GF for this functional can be safely written down over the entire $\mathcal{P}_2$ space.

The Lipschitz continuity can be translated to the $L_2(\nu;\Omega)$ space using the lifted function.
\begin{corollary}\label{cor:reg_KL_L_Lip}
For every $X\in L_2(\nu;\Omega)$, denote $G[X](z)\in L_2(\nu;\Omega)$ the functional derive defined as in~\eqref{eqn:L-grad-def} that in induced by $\tilde{F}^\epsilon$, the lifting of $F^\epsilon$, then the map is Lipschitz, in the sense that:
\[
\begin{aligned}
\int \left\Vert G[X](z) - G[Y](z) \right\Vert_2^2 d \nu(z)\leq C^2 \int \left\Vert X(z) - Y(z) \right\Vert^2  d \nu(z) \,.
\end{aligned}
\]
where the Lipschitz constant $C=C_1 + \frac{3}{\epsilon}e^{\frac{8R_0^2}{\epsilon}}$.
\end{corollary}
\begin{proof}
This is complete with brute-force calculation. For $X,Y\in L^2(\nu;\mathcal{C})$,
$$\begin{aligned}
& \int \left\Vert G[X](z) - G[Y](z) \right\Vert_2^2 d \nu(z) \\
= & \int \left\Vert \left.\nabla\tfrac{\delta F^\epsilon}{\delta\rho}\right|_{X_\# \nu}(X(z)) - \left.\nabla\tfrac{\delta F^\epsilon}{\delta\rho}\right|_{Y_\# \nu}(Y(z)) \right\Vert_2^2 d \nu(z) \\
= & \int_{\mathcal{C} \times \mathcal{C}} \left\Vert \left.\nabla\tfrac{\delta F^\epsilon}{\delta\rho}\right|_{X_\# \nu}(x) - \left.\nabla\tfrac{\delta F^\epsilon}{\delta\rho}\right|_{Y_\# \nu}(y) \right\Vert_2^2 d \left(\left( X \times Y \right)_\# \nu \right)\left(x,y\right) \\
\leq & \left(C_1 + \frac{3}{\epsilon}e^{\frac{8R_0^2}{\epsilon}}  \right)^2 \int_{\mathcal{C} \times \mathcal{C}} \left\Vert x - y \right\Vert^2  d \left(\left( X \times Y \right)_\# \nu \right)\left(x,y\right) \\
= & \left(C_1 + \frac{3}{\epsilon}e^{\frac{8R_0^2}{\epsilon}}  \right)^2 \int \left\Vert X(z) - Y(z) \right\Vert^2  d \nu(z) \, ,
\end{aligned}$$
where the first equality uses~\eqref{eqn:L-grad-def-equal} that rewrites $G[X](z) = \partial_\rho F^\epsilon[X_\#\nu]\circ X$, and Theorem~\ref{thm:reg_KL_grad_coincide}, the second equality uses $(X \times Y)_\# \nu \in \Pi(X_\# \nu, Y_\# \nu)$, and the first inequality applies~\eqref{eqn:reg_KL_grad_transport_Lip} in Lemma~\ref{lem:reg_reg_KL_first_var}.
\end{proof}

\subsection{Projected Gradient Descent for $F^\epsilon$}\label{sec:PGD_all}
The pitfall discussed in Section~\ref{sec:counterexamples} reveals the stringent requirement for the regularity of the energy functional. 
Now with the new functional being smooth throughout $\mathcal{P}_2(\Omega)$, we claim the optimization solver should finally operate properly.

The solver is designed with no surprise. 
Considering that we are in $\Omega=\mathcal{C}$, the GD is modified to fit the compact domain requirement, and we term it projected GD:
\begin{equation}\label{eqn:PGD_def_measure}
\rho_{n+1} := \operatorname{proj}_{\mathcal{C}}\left(  \left( \operatorname{Id} - h_n \left.\nabla\tfrac{\delta F^\epsilon}{\delta\rho}\right|_{\rho_n} \right)_\# \rho_n\right) \, .
\end{equation}
We have the following theorem for the convergence.

\begin{theorem}\label{prop:W2_PGD_rate}
Let $F^\epsilon$ be defined as in~\eqref{eqn:def_reg_KL}, with the target distribution $\rho^*$ having a potential that satisfies Assumption~\ref{asp:potential_function}. Let $L \in (0,+\infty)$ be the Lipschitz constant of $G[X]$ on $L^2(\nu;\mathcal{C})$ specified by Corollary~\ref{cor:reg_KL_L_Lip}. Let $\{\rho_n\}$ be the iterates generated by~\eqref{eqn:PGD_def_measure} with constant step size $h \in (0,\frac{1}{L}]$.
Denote the optimal value $F_{\mathrm{opt}}^\epsilon := \inf_{\rho \in \mathcal{P}_2(\mathcal{C})}F^\epsilon[\rho]$ and let $\rho_\mathrm{opt}$ be one minimizer so that $F^\epsilon[\rho_\mathrm{opt}]=F^\epsilon_\mathrm{opt}$. Then \begin{equation}\label{eqn:nonconvex_decay}
F^\epsilon[\rho_{n+1}] \leq F^\epsilon[\rho_n] \,, \quad \forall n \geq 0 \, ,
\end{equation} and 
\begin{equation}\label{eqn:nonconvex_rate}
\min_{0 \leq k \leq n-1}W_2(\rho_k, \rho_{k+1}) \leq \sqrt{\frac{2 \left( F^\epsilon[\rho_0] - F_\mathrm{opt}^\epsilon \right)}{h n}} \, .
\end{equation}
If, in addition, $F^\epsilon$ is geodesically convex, we obtain the sublinear convergence rate \begin{equation}\label{eqn:convex_rate}
F^\epsilon[\rho_n] - F_\mathrm{opt}^\epsilon \leq \frac{W_2^2(\rho_n, \rho_\mathrm{opt})}{2 n h} \, .
\end{equation}
Finally, if $F^\epsilon$ is further geodesically $m$-strongly convex for some $m>0$, the rate improves to
\begin{equation}\label{eqn:strongly_convex_rate}
W_2^2(\rho_n, \rho_\mathrm{opt}) \leq \left( 1 - m h\right)^{n} W_2^2(\rho_0, \rho_\mathrm{opt}) \, .
\end{equation}
\end{theorem}

The main part of the proof is to deploy equivalence established in Theorem~\ref{thm:reg_KL_grad_coincide} that builds the connection of W-derivative with L-derivative defined over $L_2(\nu;\Omega)$. The entire proof is thus translated from updating $\rho_n$ to updating $X_n$ that are connected by $\rho_n=\left(X_n\right)_\#\nu$. Since the sequence $\{X_n\}$ lives in a linear space, the updates is successful if we follow the standard PGD results under smoothness assumptions in Hilbert spaces~\cite{CP12}, \cite{CJ04}, \cite{WR22}. Corollary~\ref{cor:reg_KL_L_Lip} exactly provide this smoothness. 

\begin{proof}
Denote $\tilde{F}^\epsilon$ the lifting of $F^\epsilon$ according to~\eqref{eqn:def_lift}, we are to conduct the proof over $L^2(\nu;\mathcal{C})$. More specifically, define the updates in $L^2(\nu;\mathcal{C})$:
\begin{equation}\label{eqn:PGD_def_hilbert}
X_{n+1} = \operatorname{proj}_{\mathcal{C}}\left( X_n - h_n \left.\nabla\tfrac{\delta F^\epsilon}{\delta\rho}\right|_{(X_n)_\# \nu} \circ X_n\right)\,.
\end{equation}
We claim, if $(X_0)_\# \nu = \rho_0$, then $\rho_n=\mathrm{Law}(X_n)=(X_n)_\# \nu$. This can be seen from induction. Namely, suppose $(X_n)_\# \nu=\rho_n$:
\[
\begin{aligned}
(X_{n+1})_\#\nu 
& = (\operatorname{proj}_{\mathcal C}\!\circ (\,\operatorname{Id}-h_n\nabla_{\!}\tfrac{\delta F^\epsilon}{\delta\rho}\big|_{\rho_n}\,))_\#(X_n)_\#\nu \\
& =\operatorname{proj}_{\mathcal C\,\#}\!\Bigl((\,\operatorname{Id}-h_n\nabla_{\!}\tfrac{\delta F^\epsilon}{\delta\rho}\big|_{\rho_n}\,)_\#(X_n)_\#\nu\Bigr)\\
&= \operatorname{proj}_{\mathcal C\,\#}\!\Bigl((\,\operatorname{Id}-h_k\nabla_{\!}\tfrac{\delta F^\epsilon}{\delta\rho}\big|_{\rho_k}\,)_\#\rho_k\Bigr) =\rho_{n+1} \,,
\end{aligned}
\]
where the first and last line use the definition of $X_{n+1}$ and $\rho_{n+1}$ in~\eqref{eqn:PGD_def_hilbert} and~\eqref{eqn:PGD_def_measure} respectively.

This connection nicely translates the PGD updates in a Wasserstein space to that in a Hilbert (hence linear) space, equating:
\[
F^\epsilon[\rho_n]=\tilde{F}^\epsilon[X_n]\,,\quad\text{and that}\quad F_\mathrm{opt}^\epsilon =\tilde{F}_\mathrm{opt}^\epsilon\,,\quad \rho_{\mathrm{opt}}=\left(X_{\mathrm{opt}}\right)_\#\nu \,.
\]
Invoking the Hilbert-space PGD results (see summary in Appendix~\ref{sec:Supplement_proof_prop_W2_PGD_rate}), we have, according to Lemma~\ref{lem:nonconvex_Hilbert_rate}:
\[
F^\epsilon[\rho_{n+1}] - F^\epsilon[\rho_n] = \tilde{F}^\epsilon[X_{n+1}] - \tilde{F}^\epsilon[X_{n}] \leq 0 \, , 
\]
and that the bound on the minimal Wasserstein step
\[\begin{aligned}
\min_{0 \leq k \leq n-1}W_2(\rho_k, \rho_{k+1}) & \leq \min_{0 \leq k \leq n-1} \| X_k - X_{k+1} \| \\
& \leq 
\sqrt{\frac{2 \left( \tilde{F}^\epsilon[X_0] - \tilde{F}_\mathrm{opt}^\epsilon \right)}{h n}} = \sqrt{\frac{2 \left( F^\epsilon[\rho_0] - F_\mathrm{opt}^\epsilon \right)}{h n}} \, ,
\end{aligned}\]
concluding the proof of
\eqref{eqn:nonconvex_decay} and \eqref{eqn:nonconvex_rate}.

In the geodesic case, citing~\cite[Theorem 1]{G24}, we note that the geodesic convexity of $F$ also translates to that of $\tilde{F}^\epsilon$ over $L^2(\nu;\mathcal{C})$,\footnote{With a small modification on replacing $\mathbb{R}^d$ by $\mathcal{C}$ in the proof.} hence applying Lemma~\ref{lem:convex_Hilbert_rate}, we have:
\[
F^\epsilon[\rho_n] - F_\mathrm{opt}^\epsilon = \tilde{F}^\epsilon(X_n) - \tilde{F}_\mathrm{opt}^\epsilon \leq \frac{\left\| X_{0} - X_\mathrm{opt} \right\|^2}{2 n h} \, .
\]
Noting the requirement for $X_0$ and $X_\mathrm{opt}$ are arbitrary as long as $(X_0)_\# \nu = \rho_0$ and $(X_\mathrm{opt})_\# \nu = \rho_\mathrm{opt}$, we are free to set them to the choice that minimizes their $L_2$ distance. Namely:
\[
F^\epsilon[\rho_n] - F_\mathrm{opt}^\epsilon \leq \min_{\substack{ (X_0)_\# \nu = \rho_0, \\
 (X_\mathrm{opt})_\# \nu = \rho_\mathrm{opt}} }\left\{ \frac{\left\| X_{0} - X_\mathrm{opt} \right\|^2}{2 n h} \right\} = \frac{W_2^2(\rho_0, \rho_\mathrm{opt})}{2 n h} \,,
\]
where the last equation is provided in Lemma~\ref{lem:W2_dist_L2_char}, finishing the proof of~\eqref{eqn:convex_rate}.

Similar strategy is applied to strongly geodesic convex case. Using Lemma~\ref{lem:strongly_convex_Hilbert_rate}:
\[
W_2^2(\rho_n, \rho_\mathrm{opt}) \leq \|X_n - X_\mathrm{opt} \|^2 \leq \left( 1 - m h\right)^{n} \| X_0 - X_\mathrm{opt}\|^2 \,,
\]
and choosing the optimal pairing:
\[
\begin{aligned}
W_2^2(\rho_n, \rho_\mathrm{opt}) & \leq \min_{\substack{ (X_0)_\# \nu = \rho_0, \\
 (X_\mathrm{opt})_\# \nu = \rho_\mathrm{opt}} }\left\{\left( 1 - m h\right)^{n} \| X_0 - X_\mathrm{opt}\|^2\right\} \\
& = \left( 1 - m h\right)^{n}W_2^2(\rho_0, \rho_\mathrm{opt}) \, ,
\end{aligned}
\]
completing the proof of~\eqref{eqn:strongly_convex_rate}.

\end{proof}

\subsection{Numerical experiments}\label{sec:regkl_exp}

We implement the PGD scheme~\eqref{eqn:PGD_def_measure} for regularized KL in a two-dimensional setting.
The computational domain is the smooth bounded set $\mathcal{C}=\mathbb{B}_R(0)$ with radius $R=3$.
The target potential is chosen as a correlated quadratic form,
\[
U(x)=\frac{1}{2} x^\top A x,\qquad
A=\begin{bmatrix}1 & 1.4\\ 1.4 & 4\end{bmatrix},
\]
corresponding to a truncated Gaussian distribution on $\mathcal{C}$. The regularization kernel is Gaussian, $\varphi_\varepsilon(x)=\exp(-\frac{\|x\|^2}{2\varepsilon})$, with $\varepsilon=0.1$. For spatial discretization, we use Monte Carlo solver, and represent $\rho$ using $N=2000$ particles (see~\eqref{eqn:GF_particle}). For time discretization, we use constant step size $h=0.05$ and run PGD iteration for $100$ steps.
The initialization $\rho_0$ is a truncated anisotropic Gaussian on $\mathcal{C}$, with
\[
\rho_0(x) \;\propto\;
\exp\!\Big(-\frac{1}{2} (x-\mu_0)^\top \Sigma_0^{-1} (x-\mu_0)\Big)
\,\mathbf{1}_{\{\|x\| \leq 3\}},
\qquad
\mu_0=\begin{bmatrix} -1 \\ -2 \end{bmatrix},\quad
\Sigma_0=\begin{bmatrix}  0.5 & 0 \\ 0 & 1 \end{bmatrix}.
\]
To obtain this, we draw samples from $\mathcal{N}(\mu_0,\Sigma_0)$ and rejected the sample if $\|x\|>3$. This initialization produces an off-center, anisotropic cloud located near the lower-left boundary of $\mathcal{C}$, serving as a nontrivial starting point far from equilibrium.

\begin{figure}[htb]
  \centering
  \includegraphics[width=0.48\textwidth]{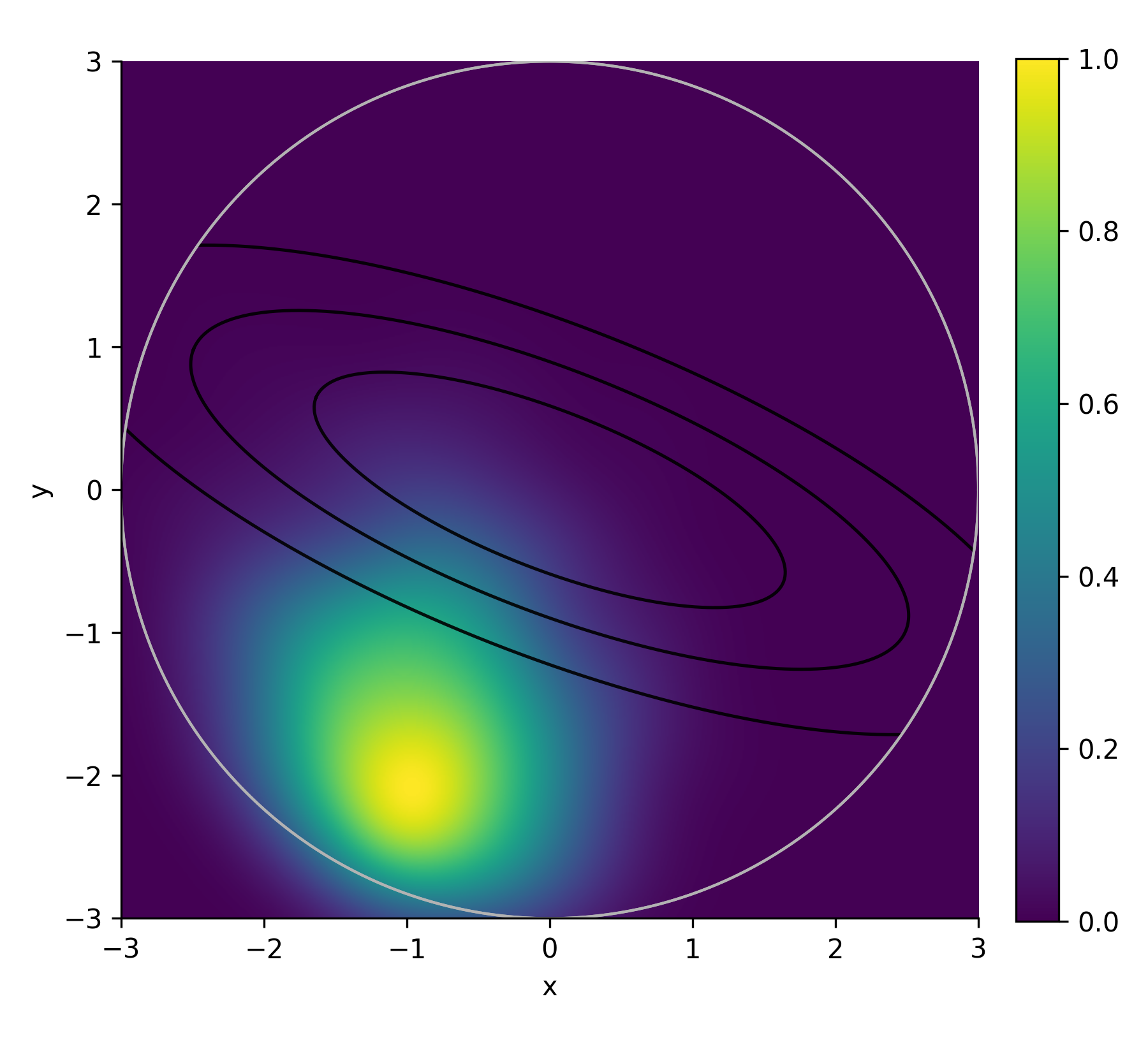}
  \includegraphics[width=0.48\textwidth]{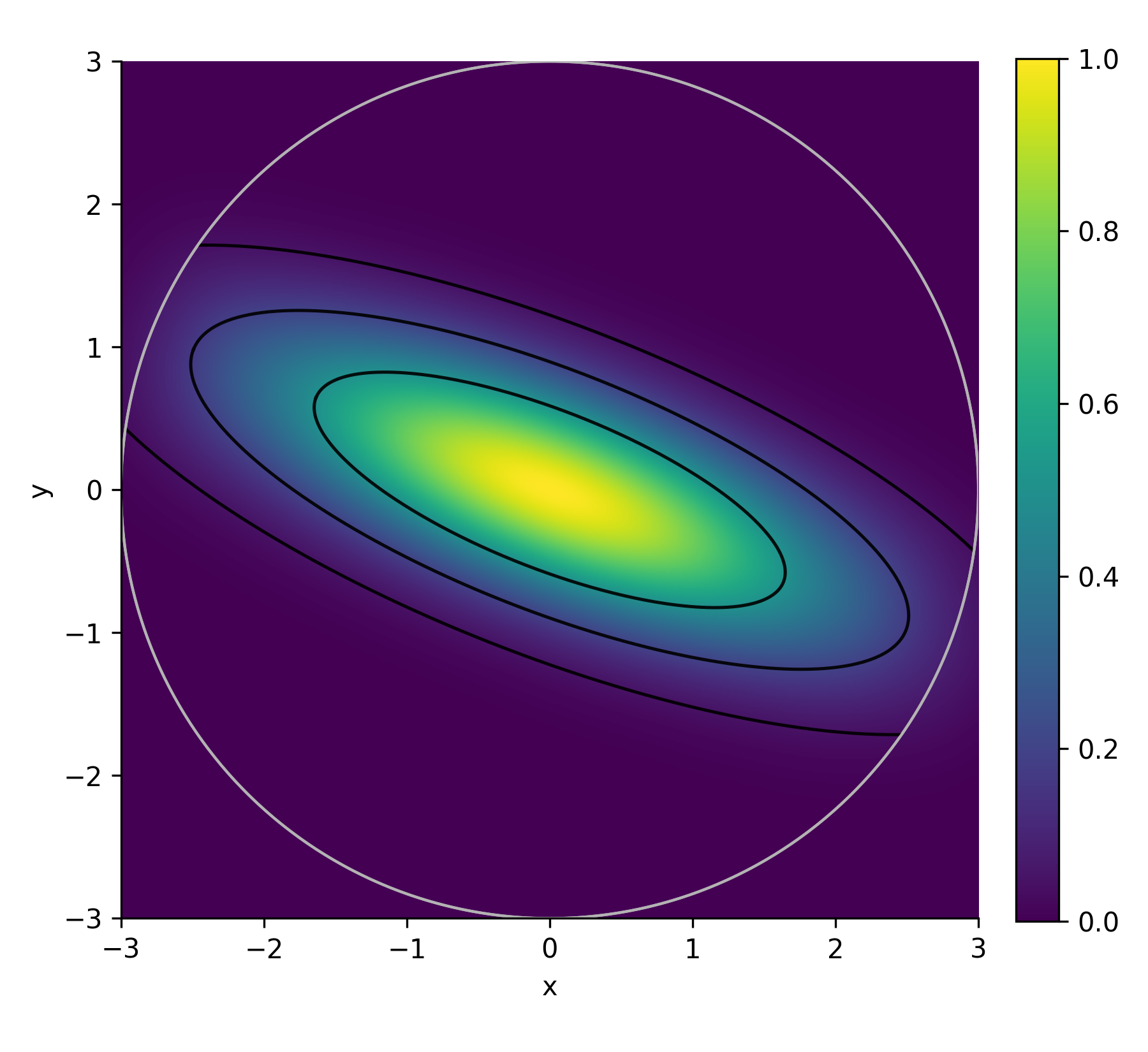}
  \caption{The initial and final distribution (after $100$ iterations) computed using kernel density estimates of $2000$ samples that runs PGD iterates~\eqref{eqn:PGD_def_measure} for the regularized KL. The colored shading shows the empirical particle density (kernel density estimate), overlaid with the target distribution's contours (50\%, 80\%, 95\%) in black. The outer white circle indicates the boundary of $\mathbb{B}_3$. The final distribution aligns well with the target distribution.}
  \label{fig:regkl_shaded}
\end{figure}

\begin{figure}[htb]
  \centering
  \includegraphics[width=0.8\textwidth]{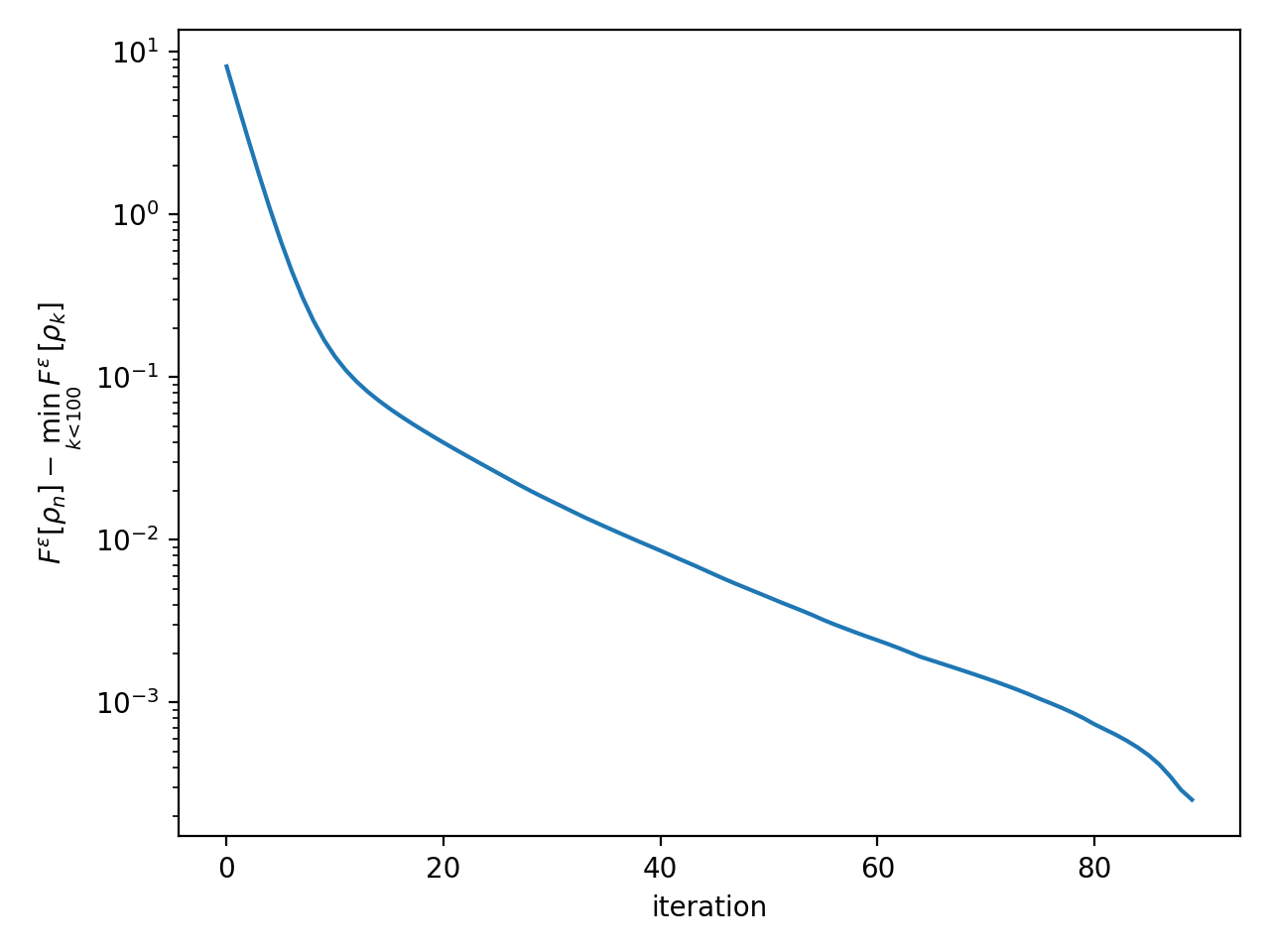}
  \caption{Evolution of the regularized energy $F^\varepsilon[\rho_n]-\min_{k \leq 100}F^\varepsilon[\rho_k]$ with step size $h=0.05$ and kernel width $\varepsilon=0.1$ for the first $90$ iterations. The exponential convergence rate confirms the dissipative behavior predicted by Theorem~\ref{prop:W2_PGD_rate}.}
  \label{fig:regkl_energy}
\end{figure}

In Figure~\ref{fig:regkl_shaded} we show the kernel density estimates of the particle distribution at the initial and final iterations. The black lines show the level set contours of the target distribution, and the white lines show the boundary of the domain. The numerical solution given by PGD remains confined within the spherical domain and approaches the target distribution as iteration progresses, a phenomenon consistent with the theoretical convergence guarantees. 
In Figure~\ref{fig:regkl_energy}, the energy of the regularized KL flow decreases steadily over iterations.  
The semi-log representation emphasizes an initial phase where the approximate optimality gap 
$F^\varepsilon[\rho_n]-F^\varepsilon[\rho^\ast]$ (with $F^\varepsilon[\rho^\ast]$ approximated by $\min_{k \leq 100}F^\varepsilon[\rho_k]$)
decays nearly exponentially.
This behavior aligns with the predicted monotone dissipation and asymptotic convergence properties of Theorem~\ref{prop:W2_PGD_rate}.

\section{Conclusion and Future Work}
Contrary to intuition, the straightforward Forward Euler (FE) discretization fails to provide a reliable numerical scheme for the gradient flow of the Kullback-Leibler (KL) divergence. Through two explicit counterexamples, we demonstrated that the loss of regularity across iterations leads to the breakdown of the method. To address this issue, we proposed applying FE not to the original KL functional, but to a regularized surrogate. When restricted to probability measures supported on a bounded convex domain, this regularization restores the necessary smoothness, and projected gradient descent (PGD) converges to a stationary point, and a minimizer whenever the regularized KL is geodesically convex.

Two important questions remain open:

\begin{enumerate}
\item {\bf Discretization error}. Even when regularity is restored, it remains unclear whether Forward Euler achieves the expected $O(h)$ accuracy as a time discretization scheme. A precise error analysis is still missing. In particular, the projection step may introduce singularities that further deteriorate the expected first-order convergence.

\item {\bf Convergence of PGD to KL minimizers}. While the regularized KL is a small perturbation of the original KL functional, it is not yet known whether minimizers of the regularized problem remain close to the true KL minimizer. Existing results, such as~\cite[Theorem~4.5]{Carrillo2019}, can only find a subsequence of the convergence as the regularization parameter vanishes. Moreover, the regularized KL may lose geodesic convexity even when $U$ is strongly convex, in which case global minimizers may fail to exist.
\end{enumerate}
It is tempting to expect positive answers to both questions, but a rigorous justification is still lacking. A thorough resolution of these issues is essential for establishing Forward Euler-type discretizations as reliable tools for optimization problems governed by Wasserstein gradient flows.

\section*{Acknowledgments}
Q.L. thanks Dr. Youssef Marzouk and Dr. Philippe Rigollet for insightful discussions. Some of the material in this paper first appeared in the unpublished notes~\cite{XL24}.
The authors acknowledge the use of chatbots for assistance with language and \LaTeX{} style.

\appendix

\section{Proof for Proposition~\ref{prop:ex2_property}}\label{sec:Appendix_proof_prop:ex2_property}

\begin{proof}[Proof for Proposition~\ref{prop:ex2_property}]
To show the first bullet point, we compute $p_1$ explicitly. Following~\eqref{eqn:KL_gradient}, the Wasserstein gradient at $\rho_0$ is
\[
\left.\nabla \frac{\delta F}{\delta \rho}\right|_{\rho_0}(x) = \nabla U(x) -\nabla V_0(x) = \begin{cases}
x+1 & x \in (-\infty,-1] \\
0 & x \in (-1,1) \\
x-1 & x \in [1,\infty)
\end{cases}\,,
\]
meaning the pushforward map is
$$T_0(x) = x - h_0 \left.\nabla \frac{\delta F}{\delta \rho}\right|_{\rho_0}(x) = \begin{cases}
(1 - h_0)x - h_0 & x \in (-\infty,-1] \\
x & x \in (-1,1) \\
(1 - h_0)x + h_0 & x \in [1,\infty)
\end{cases} \,,
$$
for step-size $h_0$. The inverse and the Jacobian can also be computed explicitly:
\[
T_0^{-1}(x) = \begin{cases}
\frac{x + h_0}{1 - h_0}\\
x\\
\frac{x - h_0}{1 - h_0}
\end{cases} \,,\quad \text{J}_{T_0^{-1}}(x) = \begin{cases}
\frac{1}{1 - h_0} & x \in (-\infty,-1] \\
1 & x \in (-1,1) \\
\frac{1}{1 - h_0} & x \in [1,\infty)
\end{cases} \ .
\]
Applying~\eqref{eqn:pushforward_density} gives
\[
\begin{aligned}
\rho_1(x) &= \rho_0(T_0^{-1}(x))\vert \J_{T_0^{-1}}(x)\vert \\
&= \frac{1}{D_0}\begin{cases}
\frac{1}{1 - h_0}\exp(-\frac{x+h_0}{1-h_0}+\frac{1}{2}) & x \in (-\infty,-1] \\
\exp(-\frac{x^2}{2}) & x \in (-1,1) \\
\frac{1}{1 - h_0}\exp(-\frac{x-h_0}{1-h_0}+\frac{1}{2}) & x \in [1,\infty)
\end{cases} \,.
\end{aligned}
\]
On each of the intervals $(-\infty,-1]$, $(-1,1)$, $[1,\infty)$, $\rho_1$ is $\mathcal{C}^\infty$, and thus $\rho_1 \in S_F^{\mathrm{f}}$ by~\eqref{eqn:KL_gradient}.
It is also immediate that $\rho_1$ has jump discontinuity at $\{-1,1\}$ two points as long as $h_0 > 0$, and therefore it is not in $W_{\text{loc}}^{1,1}(\mathbb{R})$. Hence $F$ is not Wasserstein differentiable at $\rho_1$ according to Theorem~\ref{thm:subdifferentiability_entropy}.

To prove the second bullet point, we first notice that due to the symmetry, we only need to show the validity of this formula for $x\geq 0$. When $k=0$, from the definition in~\eqref{eqn:def_p0_ex2}, $a_0 = 1, b_0 = \frac{1}{2} - \ln(D_0), c_0 = 1$. Suppose that the claim is true for $k$, then using~\eqref{eqn:KL_gradient} and~\eqref{eqn:GF_dis}:
\[
T_n(x) = x - h_n \left.\nabla \frac{\delta F}{\delta \rho}\right|_{\rho_n}(x) = \begin{cases}
x & x \in [0,1) \\
(1 - h_n)x + a_n h_n & x \in [c_n,\infty)
\end{cases} \,.
\]
Noting the monotonicity of $T_n$ allows us to compute the image of $T$ explicitly. In particular:
\[
\text{Imag}(\left.T_n\right|_{[c_n,\infty)}) = [c_{n+1},\infty)\,,\quad\text{with}\quad c_{n+1} = (1 - h_n)c_n + a_n h_n \geq 1 + (a_n - 1) h_n \geq 1\,.
\]
As a consequence, inverse and Jacobian can be computed:
\[
T_n^{-1}(x) = \begin{cases}
x & x \in [0,1) \\
\frac{x - a_n h_n}{1 - h_n} & x \in [c_{n+1},\infty)
\end{cases} \,,\quad \ \text{J}_{T_n^{-1}}(x) = \begin{cases}
1 & x \in [0,1) \\
\frac{1}{1 - h_n} & x \in [c_{n+1},\infty)
\end{cases} \,,
\]
leading to
\[
\rho_{n+1}(x)  = \begin{cases}
\frac{1}{D_0}\exp(-\frac{x^2}{2}) & x \in [0,1) \\
0 & x \in [1,c_{n+1}) \\
\exp(-a_{n+1} x + b_{n+1}) & x \in [c_{n+1}, +\infty)\\
\rho_{n+1}(-x) & x \in (-\infty,0)
\end{cases} \,,
\]
where $a_n$ and $b_n$ solve:
\[
\begin{aligned}
\exp(-a_{n+1}x+b_{n+1}) & = \rho_{n}(T_{n}^{-1}(x)) \vert \J_{T_{n}^{-1}}(x) \vert \\
& = \frac{1}{1 - h_n}\exp(-a_n\frac{x-a_n h_n}{1 - h_n} + b_n) \\
& = \exp\left(-\frac{a_n}{1 - h_n} x + b_n + \frac{a_n^2 h_n}{1 - h_n} - \ln(1 - h_n)\right) \,,
\end{aligned}
\]
meaning
\[
b_{n+1} = b_n + \frac{a_n^2 h_n}{1 - h_n} - \ln(1 - h_n)\,,\quad a_{n+1} = \frac{a_n}{1 - h_n} > a_n \geq 1\,.
\]
This finishes the proof of induction.
\end{proof}

\section{Proof of Lemma~\ref{lem:reg_reg_KL_first_var}}\label{sec:Supplement_proof_regularized_KL}
This section of appendix is dedicated to prove Lemma~\ref{lem:reg_reg_KL_first_var} that states the Lipschitz continuity for the first variation and the derivative of the first variation.

As a preparation, we first spell out the explicit form of the first variation and its gradient for~\eqref{eqn:def_reg_KL}.

\begin{proposition}
For any $\rho \in \mathcal{P}_2(\mathcal{C})$ and any $x \in \mathcal{C}$, we have \begin{equation}\label{eqn:reg_KL_first_var}
\left.\tfrac{\delta F^\epsilon}{\delta\rho}\right|_{\rho}(x) = U(x)
+ \int_{\mathcal{C}} \frac{\varphi_\epsilon(x - z)}{\int_{\mathcal{C}} \varphi_\epsilon(z - w) \, d\rho(w)} \, d\rho(z)
+ \ln\left( \int_{\mathcal{C}} \varphi_\epsilon(x - z) \, d\rho(z) \right) \, ,
\end{equation}
and
\begin{equation}\label{eqn:reg_KL_grad_first_var}
\left.\nabla\tfrac{\delta F^\epsilon}{\delta\rho}\right|_{\rho}(x) = \nabla U(x) + \int_{\mathcal{C}} \frac{\nabla \varphi_\epsilon(x-z)}{\int_{\mathcal{C}}\varphi_\epsilon(z - w) d \rho(w)} d \rho (z) + \frac{\int_{\mathcal{C}} \nabla \varphi_\epsilon(x-z)d\rho(z)}{\int_{\mathcal{C}} \varphi_\epsilon(x-w)d\rho(w)} \, .
\end{equation}
\end{proposition}
It is straightforward derivation, partially following~\cite[Proposition 3.12]{Carrillo2019}.

We are now ready to show the estimates in Lemma~\ref{lem:reg_reg_KL_first_var}.

\begin{proof}
The proof for all three statements~\eqref{eqn:reg_KL_first_var_lip},~\eqref{eqn:reg_KL_grad_first_var_lip}, and~\eqref{eqn:reg_KL_grad_transport_Lip} are similar. We give details for estimating~\eqref{eqn:reg_KL_first_var_lip} and only show the roadmap for the other two. 

By~\eqref{eqn:reg_KL_first_var} and triangular inequality,
\begin{equation}\label{eqn:reg_reg_KL_first_var_main}
\begin{aligned}
& \left\vert \left.\nabla\tfrac{\delta F^\epsilon}{\delta\rho}\right|_{\rho_1}(x) - \left.\nabla\tfrac{\delta F^\epsilon}{\delta\rho}\right|_{\rho_2}(y) \right\vert \\
\leq & \left\vert U(x) - U(y) \right\vert \\
& + \left\vert \int_{\mathcal{C}} \frac{\varphi_\epsilon(x - z)}{\int_{\mathcal{C}} \varphi_\epsilon(z - w) \, d\rho_1(w)} \, d\rho_1(z) - \int_{\mathcal{C}} \frac{\varphi_\epsilon(y - z)}{\int_{\mathcal{C}} \varphi_\epsilon(z - w) \, d\rho_2(w)} \, d\rho_2(z) \right\vert \\ 
& + \left\vert \ln\left( \int_{\mathcal{C}} \varphi_\epsilon(x - z) \, d\rho_1(z) \right) - \ln\left( \int_{\mathcal{C}} \varphi_\epsilon(y - z) \, d\rho_2(z) \right) \right\vert \\
\leq & \left\vert U(x) - U(y) \right\vert \\
& + \left\vert \int_{\mathcal{C}} \frac{\varphi_\epsilon(x - z)}{\int_{\mathcal{C}} \varphi_\epsilon(z - w) \, d\rho_1(w)} \, d\rho_1(z) - \int_{\mathcal{C}} \frac{\varphi_\epsilon(x - z)}{\int_{\mathcal{C}} \varphi_\epsilon(z - w) \, d\rho_2(w)} \, d\rho_1(z) \right\vert \\
& + \left\vert \int_{\mathcal{C}} \frac{\varphi_\epsilon(x - z)}{\int_{\mathcal{C}} \varphi_\epsilon(z - w) \, d\rho_2(w)} \, d\rho_1(z) - \int_{\mathcal{C}} \frac{\varphi_\epsilon(x - z)}{\int_{\mathcal{C}} \varphi_\epsilon(z - w) \, d\rho_2(w)} \, d\rho_2(z) \right\vert \\
& + \left\vert \int_{\mathcal{C}} \frac{\varphi_\epsilon(x - z)}{\int_{\mathcal{C}} \varphi_\epsilon(z - w) \, d\rho_2(w)} \, d\rho_2(z) - \int_{\mathcal{C}} \frac{\varphi_\epsilon(y - z)}{\int_{\mathcal{C}} \varphi_\epsilon(z - w) \, d\rho_2(w)} \, d\rho_2(z) \right\vert \\ 
& + \left\vert \ln\left( \int_{\mathcal{C}} \varphi_\epsilon(x - z) \, d\rho_1(z) \right) - \ln\left( \int_{\mathcal{C}} \varphi_\epsilon(y - z) \, d\rho_2(z) \right) \right\vert\\
=&\mathrm{Term I}+\mathrm{Term II}+\mathrm{Term III}+\mathrm{Term IV}+\mathrm{Term V}\,,
\end{aligned}
\end{equation}
We now analyze these terms separately.

\textbf{Term I}
This comes directly from Assumption~\ref{asp:potential_function}, \begin{equation}\label{eqn:reg_reg_KL_first_var_first}
\mathrm{Term I} \leq \left( C_0 + \frac{C_1R_0^2}{2} \right)\left\Vert x - y \right\Vert_2 \,.
\end{equation}

\textbf{Term II}
For the second term, \begin{equation}\label{eqn:reg_reg_KL_first_var_second}
\begin{aligned}
& \mathrm{Term II}\\
\leq &  \int_{\mathcal{C}} \left\vert \frac{\varphi_\epsilon(x - z)}{\int_{\mathcal{C}} \varphi_\epsilon(z - w) \, d\rho_1(w)} - \frac{\varphi_\epsilon(x - z)}{\int_{\mathcal{C}} \varphi_\epsilon(z - w) \, d\rho_2(w)} \right\vert d\rho_1(z)\\
= &  \int_{\mathcal{C}} \left\vert \left( \frac{\int_{\mathcal{C}} \varphi_\epsilon(z - w) \, d\rho_1(w) - \int_{\mathcal{C}} \varphi_\epsilon(z - w) \, d\rho_2(w)}{\left(\int_{\mathcal{C}} \varphi_\epsilon(z - w) \, d\rho_1(w)\right)\left(\int_{\mathcal{C}} \varphi_\epsilon(z - w) \, d\rho_2(w)\right)} \right) \varphi_\epsilon(x - z) \right\vert d\rho_1(z) \\
\leq & e^{\frac{4R_0^2}{\epsilon}}  \int_{\mathcal{C}} \left\vert \int_{\mathcal{C}} \varphi_\epsilon(z - w) \, d\rho_1(w) - \int_{\mathcal{C}} \varphi_\epsilon(z - w) \, d\rho_2(w)\right\vert d\rho_1(z) \\
\leq & \frac{e^{\frac{4R_0^2}{\epsilon}}}{\sqrt{e \epsilon}}  \int_{\mathcal{C}} W_1\left(\rho_1, \rho_2 \right) d\rho_1(z) \\
\leq & \frac{e^{\frac{4R_0^2}{\epsilon}}}{\sqrt{e \epsilon}} W_2\left(\rho_1, \rho_2 \right)  \, ,
\end{aligned}
\end{equation} where in the second inequality we used Lemma~\ref{lem:mollifier_int_bounds}; in the third inequality we used Lemma~\ref{lem:mollifier_int_bounds} and Corollary~\ref{cor:kant_rub_bd}; and in the last inequality we applied Lemma~\ref{lem:W1_W2_order}.

\textbf{Term III:} For the third term, first note that for any $p, q \in \mathcal{C}$,  \begin{equation}\label{eqn:reg_reg_KL_first_var_third_Lip}
\begin{aligned}
& \left\vert \frac{\varphi_\epsilon(x - p)}{\int_{\mathcal{C}} \varphi_\epsilon(p - w) \, d\rho_2(w)} - \frac{\varphi_\epsilon(x - q)}{\int_{\mathcal{C}} \varphi_\epsilon(q - w) \, d\rho_2(w)} \right\vert \\
\leq & \left\vert \frac{\varphi_\epsilon(x - p)}{\int_{\mathcal{C}} \varphi_\epsilon(p - w) \, d\rho_2(w)} - \frac{\varphi_\epsilon(x - p)}{\int_{\mathcal{C}} \varphi_\epsilon(q - w) \, d\rho_2(w)} \right\vert \\
& + \left\vert \frac{\varphi_\epsilon(x - p)}{\int_{\mathcal{C}} \varphi_\epsilon(q - w) \, d\rho_2(w)} - \frac{\varphi_\epsilon(x - q)}{\int_{\mathcal{C}} \varphi_\epsilon(q - w) \, d\rho_2(w)} \right\vert \\
= & \left\vert \frac{\int_{\mathcal{C}} \varphi_\epsilon(p - w) \, d\rho_2(w) - \int_{\mathcal{C}} \varphi_\epsilon(q - w) \, d\rho_2(w)}{\left( \int_{\mathcal{C}} \varphi_\epsilon(p - w) \, d\rho_2(w) \right) \left( \int_{\mathcal{C}} \varphi_\epsilon(q - w) \, d\rho_2(w)\right)} \varphi_\epsilon(x - p) \right\vert \\
& + \left\vert \frac{\varphi_\epsilon(x - p) - \varphi_\epsilon(x - q)}{\int_{\mathcal{C}} \varphi_\epsilon(q - w) \, d\rho_2(w)}  \right\vert \\
\leq & e^{\frac{4R_0^2}{\epsilon}} \left\vert \int_{\mathcal{C}} \varphi_\epsilon(p - w) \, d\rho_2(w) - \int_{\mathcal{C}} \varphi_\epsilon(q - w) \, d\rho_2(w) \right\vert \\
& + e^{\frac{2R_0^2}{\epsilon}} \left\vert \varphi_\epsilon(x - p) - \varphi_\epsilon(x - q) \right\vert \\
\leq & \left( \frac{e^{\frac{4R_0^2}{\epsilon}} + e^{\frac{2R_0^2}{\epsilon}}}{\sqrt{\epsilon e}}   \right) \left\Vert p - q\right\Vert_2 \, ,
\end{aligned} 
\end{equation} where in the second inequality we used Lemma~\ref{lem:mollifier_int_bounds}, and in the third inequality we used Lemma~\ref{lem:mollifier_raw_bounds}. This inequality, when plugged in the third term, gives: \begin{equation}\label{eqn:reg_reg_KL_first_var_third}
\begin{aligned}
& \left\vert \int_{\mathcal{C}} \frac{\varphi_\epsilon(x - z)}{\int_{\mathcal{C}} \varphi_\epsilon(z - w) \, d\rho_2(w)} \, d\rho_1(z) - \int_{\mathcal{C}} \frac{\varphi_\epsilon(x - z)}{\int_{\mathcal{C}} \varphi_\epsilon(z - w) \, d\rho_2(w)} \, d\rho_2(z) \right\vert \\
\leq & \left( \frac{e^{\frac{4R_0^2}{\epsilon}} + e^{\frac{2R_0^2}{\epsilon}}}{\sqrt{\epsilon e}}   \right) W_1(\rho_1, \rho_2) \leq \left( \frac{e^{\frac{4R_0^2}{\epsilon}} + e^{\frac{2R_0^2}{\epsilon}}}{\sqrt{\epsilon e}}   \right) W_2(\rho_1, \rho_2) \, ,
\end{aligned}
\end{equation} where in the first inequality we used~\eqref{eqn:reg_reg_KL_first_var_third_Lip} and Corollary~\eqref{cor:kant_rub_bd}; and in the second inequality we applied Lemma~\ref{lem:W1_W2_order}.

\textbf{Term IV}
For the fourth term, \begin{equation}\label{eqn:reg_reg_KL_first_var_fourth}
\begin{aligned}
& \left\vert \int_{\mathcal{C}} \frac{\varphi_\epsilon(x - z)}{\int_{\mathcal{C}} \varphi_\epsilon(z - w) \, d\rho_2(w)} \, d\rho_2(z) - \int_{\mathcal{C}} \frac{\varphi_\epsilon(y - z)}{\int_{\mathcal{C}} \varphi_\epsilon(z - w) \, d\rho_2(w)} \, d\rho_2(z) \right\vert \\
\leq & \int_{\mathcal{C}} \left\vert \frac{\varphi_\epsilon(x - z) - \varphi_\epsilon(y - z)}{\int_{\mathcal{C}} \varphi_\epsilon(z - w) \, d\rho_2(w)}  \right\vert \, d\rho_2(z) \\
\leq & e^{\frac{2R_0^2}{\epsilon}} \int_{\mathcal{C}} \left\vert \varphi_\epsilon(x - z) - \varphi_\epsilon(y - z) \right\vert \, d\rho_2(z) \\
\leq & \frac{e^{\frac{2R_0^2}{\epsilon}}}{\sqrt{\epsilon e}} \left\Vert x - y \right\Vert_2 \, ,
\end{aligned}
\end{equation} where the second inequality used Lemma~\ref{lem:mollifier_int_bounds}, and in the third inequality we applied Lemma~\ref{lem:mollifier_raw_bounds}.

\textbf{Term V}
Utilizing Lemma~\ref{lem:mollifier_int_bounds} and the fact that the derivative of $s \mapsto \ln(s)$ is bounded within $\left(0,e^{\frac{2R_0^2}{\epsilon}}\right]$ for $s \in \left[e^{\frac{2R_0^2}{\epsilon}}, +\infty\right)$, we have:
\begin{equation}\label{eqn:termV_lemma_lip}
\mathrm{Term V}\leq \exp\left( \frac{2R_0^2}{\epsilon}\right) \cdot \left\vert  \int_{\mathcal{C}} \varphi_\epsilon(x - z) \, d\rho_1(z) - \int_{\mathcal{C}} \varphi_\epsilon(y - z) \, d\rho_2(z) \right\vert \,.
\end{equation}
Furthermore,
\begin{equation}\label{eqn:reg_reg_KL_first_var_fifth}
\begin{aligned}
& \left\vert  \int_{\mathcal{C}} \varphi_\epsilon(x - z) \, d\rho_1(z) - \int_{\mathcal{C}} \varphi_\epsilon(y - z) \, d\rho_2(z) \right\vert \\
\leq & \left\vert  \int_{\mathcal{C}} \varphi_\epsilon(x - z) \, d\rho_1(z) - \int_{\mathcal{C}} \varphi_\epsilon(y - z) \, d\rho_1(z) \right\vert \\
& + \left\vert  \int_{\mathcal{C}} \varphi_\epsilon(y - z) \, d\rho_1(z) - \int_{\mathcal{C}} \varphi_\epsilon(y - z) \, d\rho_2(z) \right\vert \\
\leq & e^{\frac{2R_0^2}{\epsilon}} \left\Vert x - y \right\Vert_2 + e^{\frac{2R_0^2}{\epsilon}} W_1(\rho_1, \rho_2) \\
\leq & e^{\frac{2R_0^2}{\epsilon}} \left\Vert x - y \right\Vert_2 + e^{\frac{2R_0^2}{\epsilon}} W_2(\rho_1, \rho_2) \, ,
\end{aligned}
\end{equation} where in the first inequality we applied triangular inequality; in the second inequality we applied Lemma~\ref{lem:mollifier_raw_bounds} and Corollary~\ref{cor:kant_rub_bd}; and in the last inequality we used Lemma~\ref{lem:W1_W2_order}.

Plugging~\eqref{eqn:reg_reg_KL_first_var_first},~\eqref{eqn:reg_reg_KL_first_var_second},~\eqref{eqn:reg_reg_KL_first_var_third},~\eqref{eqn:reg_reg_KL_first_var_fourth}, and~\eqref{eqn:reg_reg_KL_first_var_fifth} back in~\eqref{eqn:reg_reg_KL_first_var_main}, we have the uniform Lipschitz condition over $\mathcal{P}_2(\mathcal{C}) \times \mathcal{C}$, and complete the proof of~\eqref{eqn:reg_KL_first_var_lip}.

The second bound~\eqref{eqn:reg_KL_grad_first_var_lip} is obtained by following similar steps in the proof of~\eqref{eqn:reg_KL_first_var_lip}.
By triangular inequality,  \begin{equation}\label{eqn:reg_reg_KL_grad_first_var_main}
\begin{aligned}
& \left\Vert \left.\nabla\tfrac{\delta F^\epsilon}{\delta\rho}\right|_{\rho_1}(x) - \left.\nabla\tfrac{\delta F^\epsilon}{\delta\rho}\right|_{\rho_2}(y) \right\Vert_2 \\
\leq & \left\Vert \nabla U(x) - \nabla U(y)\right\Vert_2 \\
& + \left\Vert \int_{\mathcal{C}} \frac{\nabla \varphi_\epsilon(x - z)}{\int_{\mathcal{C}} \varphi_\epsilon(z - w) \, d\rho_1(w)} \, d\rho_1(z) - \int_{\mathcal{C}} \frac{\nabla \varphi_\epsilon(y - z)}{\int_{\mathcal{C}} \varphi_\epsilon(z - w) \, d\rho_2(w)} \, d\rho_2(z) \right\Vert_2 \\
& + \left\Vert \frac{\int_{\mathcal{C}} \nabla \varphi_\epsilon(x-z)d\rho_1(z)}{\int_{\mathcal{C}} \varphi_\epsilon(x-w)d\rho_1(w)}  - \frac{\int_{\mathcal{C}} \nabla \varphi_\epsilon(y-z)d\rho_2(z)}{\int_{\mathcal{C}} \varphi_\epsilon(y-w)d\rho_2(w)}  \right\Vert_2 \\
\leq & \left\Vert \nabla U(x) - \nabla U(y)\right\Vert_2 \\
& + \left\Vert \int_{\mathcal{C}} \frac{\nabla \varphi_\epsilon(x - z)}{\int_{\mathcal{C}} \varphi_\epsilon(z - w) \, d\rho_1(w)} \, d\rho_1(z) - \int_{\mathcal{C}} \frac{\nabla \varphi_\epsilon(y - z)}{\int_{\mathcal{C}} \varphi_\epsilon(z - w) \, d\rho_1(w)} \, d\rho_1(z) \right\Vert_2 \\
& + \left\Vert \int_{\mathcal{C}} \frac{\nabla \varphi_\epsilon(y - z)}{\int_{\mathcal{C}} \varphi_\epsilon(z - w) \, d\rho_1(w)} \, d\rho_1(z) - \int_{\mathcal{C}} \frac{\nabla \varphi_\epsilon(y - z)}{\int_{\mathcal{C}} \varphi_\epsilon(z - w) \, d\rho_2(w)} \, d\rho_2(z) \right\Vert_2 \\
& + \left\Vert \frac{\int_{\mathcal{C}} \nabla \varphi_\epsilon(x-z)d\rho_1(z)}{\int_{\mathcal{C}} \varphi_\epsilon(x-w)d\rho_1(w)}  - \frac{\int_{\mathcal{C}} \nabla \varphi_\epsilon(y-z)d\rho_1(z)}{\int_{\mathcal{C}} \varphi_\epsilon(y-w)d\rho_1(w)}  \right\Vert_2 \\
& + \left\Vert \frac{\int_{\mathcal{C}} \nabla \varphi_\epsilon(y-z)d\rho_1(z)}{\int_{\mathcal{C}} \varphi_\epsilon(y-w)d\rho_1(w)}  - \frac{\int_{\mathcal{C}} \nabla \varphi_\epsilon(y-z)d\rho_2(z)}{\int_{\mathcal{C}} \varphi_\epsilon(y-w)d\rho_2(w)}  \right\Vert_2 \, .
\end{aligned}
\end{equation}
One then need to proceed controlling these five terms respectively as done above. We leave out the details, but mention that these controls are provided by Assumption~\ref{asp:potential_function}, Lemmas~\ref{lem:reg_KL_grad_Lip_second},~\ref{lem:reg_KL_grad_Lip_third},~\ref{lem:reg_KL_grad_Lip_fourth}, and~\ref{lem:reg_KL_grad_Lip_fifth} respectively.

The third bound~\eqref{eqn:reg_KL_grad_transport_Lip} also follows from a brute-force computation. Expanding $\left.\nabla\tfrac{\delta F^\epsilon}{\delta\rho}\right|_{\rho}$ by~\eqref{eqn:reg_KL_grad_first_var} and apply the triangle inequality:
$$\begin{aligned}
& \sqrt{\int_{\mathcal{C}^2} \left\Vert \left.\nabla\tfrac{\delta F^\epsilon}{\delta\rho}\right|_{\rho_1}(x) - \left.\nabla\tfrac{\delta F^\epsilon}{\delta\rho}\right|_{\rho_2}(y) \right\Vert_2^2 d \pi(x,y) } \\
\leq & \sqrt{\int_{\mathcal{C}^2} \left\Vert \nabla U(x) - \nabla U(y) \right\Vert_2^2 d \pi(x,y) } \\
& + \sqrt{\int_{\mathcal{C}^2} \left\Vert \int_{\mathcal{C}} \frac{\nabla \varphi_\epsilon(x-z)}{\int_{\mathcal{C}}\varphi_\epsilon(z - w) d \rho_1(w)} d \rho_1 (z) - \int_{\mathcal{C}} \frac{\nabla \varphi_\epsilon(y-z)}{\int_{\mathcal{C}}\varphi_\epsilon(z - w) d \rho_2(w)} d \rho_2 (z) \right\Vert_2^2 d \pi(x,y) } \\
& + \sqrt{\int_{\mathcal{C}^2} \left\Vert \frac{\int_{\mathcal{C}} \nabla \varphi_\epsilon(x-z)d\rho_1(z)}{\int_{\mathcal{C}} \varphi_\epsilon(x-w)d\rho_1(w)} - \frac{\int_{\mathcal{C}} \nabla \varphi_\epsilon(y-z)d\rho_2(z)}{\int_{\mathcal{C}} \varphi_\epsilon(y-w)d\rho_2(w)} \right\Vert_2^2 d \pi(x,y)} \, .
\end{aligned}$$
The first term is bounded by $$\sqrt{\int_{\mathcal{C}^2} \left\Vert \nabla U(x) - \nabla U(y) \right\Vert_2^2 d \pi(x,y) } \leq C_1\sqrt{\int_{\mathcal{C}^2} \left\Vert x - y \right\Vert_2^2 d \pi(x,y)} \, $$ thanks to Assumption~\ref{asp:potential_function}. The control for the latter two terms are provided by Lemma~\ref{lem:reg_KL_W2_Lip_first} and Lemma~\ref{lem:reg_KL_W2_Lip_second} respectively.
\end{proof}

\section{Auxiliary Lemmas}

In this section, we give a few lemmas that we use repeatedly in our proofs.

\subsection{Estimates on Gaussian Kernels}

\begin{lemma}\label{lem:mollifier_raw_bounds}
Let $x \in \mathcal{B}(0;R)$ for some $R \in (0,+\infty)$ and set $\varphi_\epsilon(x) := \exp\!\big(-\tfrac{\|x\|_2^2}{2\epsilon}\big)$.
Then $\exp(-\frac{R^2}{2\epsilon}) \leq \varphi_\epsilon(x) \leq 1 \, $, $\left\Vert \nabla \varphi_\epsilon(x) \right\Vert_2 \leq \frac{1}{\sqrt{\epsilon e}}$, and $\left\Vert \operatorname{Hess}[\varphi_\epsilon](x) \right\Vert_2 \leq \frac{2}{\epsilon e^{\frac{3}{2}}}$.
\end{lemma}

\begin{proof}
The first inequality follows from $0 \leq \frac{\Vert x \Vert_2^2}{2 \epsilon} \leq \frac{R^2}{2\epsilon}$, for $x \in \mathcal{B}(0;R)$.

For the second inequality, note $\nabla \varphi_\epsilon(x) = -\frac{x}{\epsilon}\exp(-\frac{\Vert x \Vert_2^2}{2\epsilon})$, and therefore $$\left\Vert \nabla \varphi_\epsilon(x) \right\Vert_2 = \left\Vert -\frac{x}{\epsilon}\exp(-\frac{\Vert x \Vert_2^2}{2\epsilon}) \right\Vert_2 = \frac{\left\Vert x \right\Vert_2}{\epsilon} \exp(-\frac{\Vert x \Vert_2^2}{2\epsilon}) \, .  $$
Taking the derivative, we know that the maximum of $y \mapsto \frac{y}{\epsilon} \exp(-\frac{y^2}{2\epsilon})$ on $y \in [0,+\infty)$ gives the desired result.

For the third inequality, note $\text{Hess}[\varphi_\epsilon](x) = \frac{x x^\top - \epsilon I}{\epsilon^2}\exp(-\frac{\Vert x \Vert_2^2}{2\epsilon})$.
Then $$\begin{aligned}
\left\Vert \text{Hess}[\varphi_\epsilon](x) \right\Vert_2 & = \sup_{v \in \mathbb{R}^d : \Vert v \Vert_2 = 1} v^\top \frac{x x^\top - \epsilon I}{\epsilon^2}\exp(-\frac{\Vert x \Vert_2^2}{2\epsilon}) v^\top \\
& = \sup_{v \in \mathbb{R}^d : \Vert v \Vert_2 = 1} \frac{ \left\Vert x^\top v\right\Vert_2^2 - \epsilon}{\epsilon^2}\exp(-\frac{\Vert x \Vert_2^2}{2\epsilon})  \\
& = \sup_{y \in [0,R], z \in [0, 1]} \frac{y^2 z^2 - \epsilon}{\epsilon^2}\exp(-\frac{y^2}{2\epsilon})  \\
& \leq \sup_{y \in [0,R]} \frac{y^2 - \epsilon}{\epsilon^2}\exp(-\frac{y^2}{2\epsilon})  \\
& \leq \frac{2}{\epsilon e^{\frac{3}{2}}} \, .
\end{aligned}$$
Here we used the duality definition of matrix norm. The last inequality comes from explicit computation of a one-dimensional optimization.
\end{proof}

\begin{lemma}\label{lem:mollifier_int_bounds}
Suppose $\rho \in \mathcal{P}_2(\mathbb{R}^d)$ satisfies $\text{supp}(\rho) \subseteq \mathcal{B}(0; R)$ for some $R \in (0,+\infty)$.
Then for any $y \in \mathcal{B}(0; R)$, we have $\exp(-\frac{2 R^2}{\epsilon}) \leq \int_{\mathbb{R}^d}\varphi_\epsilon(x - y) d \rho(x) \leq 1$.
In addition, for any $x,y \in \mathcal{B}(0; R)$, $\Vert \nabla \varphi_\epsilon(x - y) \Vert_2 \leq \frac{1}{\sqrt{\epsilon e}}$.
\end{lemma}
This is a direct corollary of Lemma~\ref{lem:mollifier_raw_bounds}.

\subsection{Standard Lemmas from Optimal Transport}

The following results collect standard facts from optimal transport; see~\cite{AGS08}, \cite{CD18} for detailed proofs and further discussion.

\begin{lemma}[$L^2$ characterization of the $W_2$ distance]
\label{lem:W2_dist_L2_char}
Let $\nu$ be an atomless probability measure on a Polish space, and let $\Omega$ be as specified in Section~\ref{sec:basic_def}.
For any $\rho_A,\rho_B\in\mathcal P_2(\Omega)$ we have
\[
W_2^2(\rho_A,\rho_B)
= \inf\Bigl\{ \|X-Y\|_{L^2(\nu)}^2 : X,Y\in L^2(\nu;\Omega),\ 
X_\#\nu=\rho_A,\ Y_\#\nu=\rho_B \Bigr\}.
\]
Moreover, the infimum is attained, i.e., the optimal coupling admits a representation 
$(X,Y)\in L^2(\nu;\Omega\times\Omega)$ with marginals $\rho_A$ and $\rho_B$.
\end{lemma}

\begin{lemma}[Kantorovich--Rubinstein duality]\label{lem:kant_rub_dual}
For any $\mu,\nu \in \mathcal{P}_1(\Omega)$, their Wasserstein-1 distance admits the dual representation
\[
  W_{1}(\mu,\nu)
   = \sup_{\|f\|_{\mathrm{Lip}}\le 1}
      \left( \int f\,d\mu - \int f\,d\nu \right),
\]
where the supremum is taken over all real-valued $1$-Lipschitz functions $f:\Omega \to \mathbb{R}$.
\end{lemma}

\begin{corollary}[Generalized KR bound]\label{cor:kant_rub_bd}
Let $\rho_A,\rho_B\in\mathcal P_1(\Omega)$ and $(H,\langle\cdot\rangle_H)$ be a Hilbert space with induced norm $\|\cdot\|_H$. 
If $g:\Omega \to H$ is $L$-Lipschitz with respect to the Euclidean distance on $\Omega$ and the norm $\|\cdot\|_H$, then
\[
\Big\|\int g\,d\rho_A-\int g\,d\rho_B\Big\|_H \;\leq\; L \cdot W_1(\rho_A,\rho_B).
\]
\end{corollary}

\begin{proof}
For any $v\in H$ with $\|v\|_H\leq 1$, the scalar map $x\mapsto \langle v,g(x)\rangle_H$ is $L$-Lipschitz. 
By Lemma~\ref{lem:kant_rub_dual},
\[
\big|\langle v,\int g\,d\rho_A-\int g\,d\rho_B\rangle_H\big|
\; = \; \big|\int \langle v, g \rangle_H \,d\rho_A - \int \langle v, g \rangle_H \,d\rho_B\big|
\;\leq\; L \cdot W_1(\rho_A,\rho_B).
\]
Taking the supremum over all $\|v\| \leq 1$ yields the result.
\end{proof}

\begin{lemma}[Comparison between $W_1$ and $W_2$]\label{lem:W1_W2_order}
For any $\mu,\nu\in\mathcal{P}_2(\Omega)$ one has
\[
  W_{1}(\mu,\nu)\;\leq\;W_{2}(\mu,\nu).
\]
\end{lemma}

The inequality follows directly from Jensen’s inequality applied to the definition of $W_p$ with $p=1,2$.

\subsection{Explicit Formula for the KL Divergence in Section~\ref{sec:example_2}}\label{sec:Supplement_numerical_example_2}

\begin{lemma}\label{lem:KL_formula_ex2}
Let $D_0$ be defined as in~\eqref{eqn:def_d0_reg}, and let $a, c \in [1, +\infty)$.
Consider the symmetric piecewise density
\[
\rho(x)=
\begin{cases}
\frac{1}{D_0}e^{-x^2/2}, & x\in[0,1),\\
0, & x\in[1,c),\\
e^{-a x + b}, & x\in[c,\infty),\\
\rho(-x), & x \in (-\infty, 0),
\end{cases}.
\]
Let $\rho^*(x)=\frac{1}{\sqrt{2\pi}}e^{-x^2/2}$ be the standard Gaussian.
Then the KL divergence $\KL[\rho\vert\rho^*]$ admits the decomposition
\begin{equation}\label{eqn:KL_closedform_compact}
\KL[\rho\vert\rho^*]
= \underbrace{\frac{2}{D_0}\ln\left(\frac{\sqrt{2\pi}}{D_0}\right)\int_0^1 e^{-x^2/2}\,dx}_{\text{core}}
+ \underbrace{2e^{b-ac} \Psi(a,b,c)}_{\text{tails}},
\end{equation}
where
\begin{align}\label{eqn:KL_tail_part}
\Psi(a,b,c)
& = -\left(c+\frac{1}{a}\right) + \frac{b}{a} + \frac{1}{2}\left(\frac{c^2}{a} + \frac{2c}{a^2} + \frac{2}{a^3}\right) + \frac{\ln\sqrt{2\pi}}{a}.
\end{align}
\end{lemma}

\begin{proof}
The expression follows from integrating
$\rho(x)\ln\!\frac{\rho(x)}{\rho^*(x)}$
over each part of the piecewise definition of $\rho$.

\textbf{(i) Core part $x\in[-1,1]$.}
On this interval $\rho(x)=D_0^{-1}e^{-x^2/2}$, and therefore,
\[
\int_{-1}^{1}\rho(x)\ln \frac{\rho(x)}{\rho^*(x)}\,dx
= \frac{2}{D_0} \ln\left(\frac{\sqrt{2\pi}}{D_0}\right) \int_0^1\! e^{-x^2/2}\,dx
   \, ,
\]
which gives the ``core'' part  in~\eqref{eqn:KL_closedform_compact}.

\smallskip\noindent
\textbf{(ii) Tail parts $x\ge c$ and $x\le -c$.}
For $x\ge c$ we have $\rho(x)=e^{-a x+b}$.
To evaluate the corresponding integrals we use the following elementary identities:
\begin{align*}
\int_c^\infty e^{-a x+b}\,dx &= \frac{e^{b-a c}}{a},\\
\int_c^\infty x\,e^{-a x+b}\,dx &= e^{b-a c}\frac{ac+1}{a^2},\\
\int_c^\infty x^2 e^{-a x+b}\,dx &= e^{b-a c}\!\left(
\frac{c^2}{a}+\frac{2c}{a^2}+\frac{2}{a^3}
\right).
\end{align*}
Substituting these into
\[
\int_c^\infty \rho(x)\ln \frac{\rho(x)}{\rho^*(x)}\,dx
= \int_c^\infty e^{-a x+b}
\left[\left(-a x+b\right) - \left(-\frac{x^2}{2}-\ln\sqrt{2\pi}\right)\right]dx
\]
yields the expression for $\Psi(a,b,c)$ in~\eqref{eqn:KL_tail_part}.
By symmetry, the same contribution arises from the left tail $x\leq -c$,
giving the overall factor of~$2$.
\end{proof}

\section{Detailed calculation used in proving Lemma~\ref{lem:reg_reg_KL_first_var}}\label{sec:Supplement_detail_proof_regularized_KL}

We give estimates used in the proof for Lemma~\ref{lem:reg_reg_KL_first_var}.

\subsection{Details for~\eqref{eqn:reg_KL_grad_first_var_lip}}

\begin{lemma}\label{lem:reg_KL_grad_Lip_second}
Under the setup of Lemma~\ref{lem:reg_reg_KL_first_var}, 
\[
\left\Vert \int_{\mathcal{C}} \frac{\nabla \varphi_\epsilon(x - z)}{\int_{\mathcal{C}} \varphi_\epsilon(z - w) \, d\rho_1(w)} \, d\rho_1(z) - \int_{\mathcal{C}} \frac{\nabla \varphi_\epsilon(y - z)}{\int_{\mathcal{C}} \varphi_\epsilon(z - w) \, d\rho_1(w)} \, d\rho_1(z) \right\Vert_2 \leq \frac{2e^{\frac{2R_0^2}{\epsilon} - \frac{3}{2}}}{\epsilon}  \left\Vert x - y \right\Vert_2 \, .
\]
\end{lemma}

\begin{proof}
Note \[\begin{aligned}
& \left\Vert \int_{\mathcal{C}} \frac{\nabla \varphi_\epsilon(x - z)}{\int_{\mathcal{C}} \varphi_\epsilon(z - w) \, d\rho_1(w)} \, d\rho_1(z) - \int_{\mathcal{C}} \frac{\nabla \varphi_\epsilon(y - z)}{\int_{\mathcal{C}} \varphi_\epsilon(z - w) \, d\rho_1(w)} \, d\rho_1(z) \right\Vert_2 \\
\leq & \int_{\mathcal{C}} \left\Vert \frac{\nabla \varphi_\epsilon(x - z)- \nabla \varphi_\epsilon(y - z)}{\int_{\mathcal{C}} \varphi_\epsilon(z - w) \, d\rho_1(w)} \right\Vert_2 \, d\rho_1(z)  \\ 
\leq & e^{\frac{2R_0^2}{\epsilon}} \int_{\mathcal{C}} \left\Vert \nabla \varphi_\epsilon(x - z)- \nabla \varphi_\epsilon(y - z) \right\Vert_2 \, d\rho_1(z)  \\
\leq & \frac{2e^{\frac{2R_0^2}{\epsilon} - \frac{3}{2}}}{\epsilon} \int_{\mathcal{C}} \left\Vert (x - z)- (y - z) \right\Vert_2 \, d\rho_1(z) \\
\leq & \frac{2e^{\frac{2R_0^2}{\epsilon} - \frac{3}{2}}}{\epsilon}  \left\Vert x - y \right\Vert_2 \, ,
\end{aligned}\] where the second inequality applies Lemma~\ref{lem:mollifier_int_bounds}, and the third inequality uses Lemma~\ref{lem:mollifier_raw_bounds}.
\end{proof}

\begin{lemma}\label{lem:reg_KL_grad_Lip_third}
Under the setup of Lemma~\ref{lem:reg_reg_KL_first_var}, 
\[
\left\Vert \int_{\mathcal{C}} \frac{\nabla \varphi_\epsilon(y - z)}{\int_{\mathcal{C}} \varphi_\epsilon(z - w) \, d\rho_1(w)} \, d\rho_1(z) - \int_{\mathcal{C}} \frac{\nabla \varphi_\epsilon(y - z)}{\int_{\mathcal{C}} \varphi_\epsilon(z - w) \, d\rho_2(w)} \, d\rho_2(z) \right\Vert_2 \leq \frac{e^{\frac{4R_0^2}{\epsilon}-1} + 2e^{\frac{2R_0^2}{\epsilon} - \frac{3}{2}}}{\epsilon} W_2(\rho_1, \rho_2) \, .
\]
\end{lemma}

\begin{proof}
By triangular inequality, \begin{equation}\label{eqn:lem_reg_KL_grad_Lip_third_main}
\begin{aligned}
& \left\Vert \int_{\mathcal{C}} \frac{\nabla \varphi_\epsilon(y - z)}{\int_{\mathcal{C}} \varphi_\epsilon(z - w) \, d\rho_1(w)} \, d\rho_1(z) - \int_{\mathcal{C}} \frac{\nabla \varphi_\epsilon(y - z)}{\int_{\mathcal{C}} \varphi_\epsilon(z - w) \, d\rho_2(w)} \, d\rho_2(z) \right\Vert_2 \\
\leq & \left\Vert \int_{\mathcal{C}} \frac{\nabla \varphi_\epsilon(y - z)}{\int_{\mathcal{C}} \varphi_\epsilon(z - w) \, d\rho_1(w)} \, d\rho_1(z) - \int_{\mathcal{C}} \frac{\nabla \varphi_\epsilon(y - z)}{\int_{\mathcal{C}} \varphi_\epsilon(z - w) \, d\rho_2(w)} \, d\rho_1(z) \right\Vert_2 \\
& + \left\Vert \int_{\mathcal{C}} \frac{\nabla \varphi_\epsilon(y - z)}{\int_{\mathcal{C}} \varphi_\epsilon(z - w) \, d\rho_2(w)} \, d\rho_1(z) - \int_{\mathcal{C}} \frac{\nabla \varphi_\epsilon(y - z)}{\int_{\mathcal{C}} \varphi_\epsilon(z - w) \, d\rho_2(w)} \, d\rho_2(z) \right\Vert_2 \, .
\end{aligned}
\end{equation}

For the first term, we can bound it by \begin{equation}\label{eqn:lem_reg_KL_grad_Lip_third_first}
\begin{aligned}
& \left\Vert \int_{\mathcal{C}} \frac{\nabla \varphi_\epsilon(y - z)}{\int_{\mathcal{C}} \varphi_\epsilon(z - w) \, d\rho_1(w)} \, d\rho_1(z) - \int_{\mathcal{C}} \frac{\nabla \varphi_\epsilon(y - z)}{\int_{\mathcal{C}} \varphi_\epsilon(z - w) \, d\rho_2(w)} \, d\rho_1(z) \right\Vert_2 \\
\leq & \int_{\mathcal{C}} \left\Vert \frac{\left( \int_{\mathcal{C}} \varphi_\epsilon(z - w) \, d\rho_1(w) - \int_{\mathcal{C}} \varphi_\epsilon(z - w) \, d\rho_2(w) \right) \cdot \nabla \varphi_\epsilon(y - z)}{\left(\int_{\mathcal{C}} \varphi_\epsilon(z - w) \, d\rho_1(w) \right)\left( \int_{\mathcal{C}} \varphi_\epsilon(z - w) \, d\rho_2(w)\right)}\right\Vert_2 \, d\rho_1(z) \\
\leq & e^{\frac{4R_0^2}{\epsilon}} \int_{\mathcal{C}} \left\vert \int_{\mathcal{C}} \varphi_\epsilon(z - w) \, d\rho_1(w) - \int_{\mathcal{C}} \varphi_\epsilon(z - w) \, d\rho_2(w) \right\vert  \cdot \left\Vert \nabla \varphi_\epsilon(y - z)\right\Vert_2 \, d\rho_1(z) \\
\leq & \frac{e^{\frac{4R_0^2}{\epsilon}-\frac{1}{2}}}{\sqrt{\epsilon}} \int_{\mathcal{C}} \left\vert \int_{\mathcal{C}} \varphi_\epsilon(z - w) \, d\rho_1(w) - \int_{\mathcal{C}} \varphi_\epsilon(z - w) \, d\rho_2(w) \right\vert \, d\rho_1(z) \\
\leq & \frac{e^{\frac{4R_0^2}{\epsilon}-1}}{\epsilon} \int_{\mathcal{C}} W_1(\rho_1, \rho_2) \, d\rho_1(z) \\
\leq & \frac{e^{\frac{4R_0^2}{\epsilon}-1}}{\epsilon}  W_2(\rho_1, \rho_2)  \, ,
\end{aligned}
\end{equation} where the second inequality applies Lemma~\ref{lem:mollifier_int_bounds}; the third inequalities use Lemma~\ref{lem:mollifier_raw_bounds}; the fourth inequality utilizes Lemma~\ref{lem:mollifier_raw_bounds} and Corollary~\ref{cor:kant_rub_bd}; and the last inequality applies Lemma~\ref{lem:W1_W2_order}.

To bound the second term, we first bound the Lipschitz constant of $y \mapsto \frac{\nabla \varphi_\epsilon(y - z)}{\int_{\mathcal{C}} \varphi_\epsilon(z - w) \, d\rho_2(w)}$.
For any $p, q \in \mathcal{C}$, \begin{equation}\label{eqn:lem_reg_KL_grad_Lip_third_second_Lip}
\begin{aligned}
& \left\Vert \frac{\nabla \varphi_\epsilon(p - z)}{\int_{\mathcal{C}} \varphi_\epsilon(z - w) \, d\rho_2(w)} -  \frac{\nabla \varphi_\epsilon(q - z)}{\int_{\mathcal{C}} \varphi_\epsilon(z - w) \, d\rho_2(w)}  \right\Vert_2 \\
\leq & e^{\frac{2R_0^2}{\epsilon}} \left\Vert \nabla \varphi_\epsilon(p - z) - \nabla \varphi_\epsilon(q - z)  \right\Vert_2 \\
\leq & \frac{2e^{\frac{2R_0^2}{\epsilon} - \frac{3}{2}}}{\epsilon} \left\Vert p - q  \right\Vert_2 \, ,
\end{aligned}
\end{equation} where the first inequality uses Lemma~\ref{lem:mollifier_int_bounds}, and the second inequality applies Lemma~\ref{lem:mollifier_raw_bounds}.
Now 
\begin{equation}\label{eqn:lem_reg_KL_grad_Lip_third_second}
\begin{aligned}
& \left\Vert \int_{\mathcal{C}} \frac{\nabla \varphi_\epsilon(y - z)}{\int_{\mathcal{C}} \varphi_\epsilon(z - w) \, d\rho_2(w)} \, d\rho_1(z) - \int_{\mathcal{C}} \frac{\nabla \varphi_\epsilon(y - z)}{\int_{\mathcal{C}} \varphi_\epsilon(z - w) \, d\rho_2(w)} \, d\rho_2(z) \right\Vert_2 \\
\leq & \frac{2e^{\frac{2R_0^2}{\epsilon} - \frac{3}{2}}}{\epsilon} W_1(\rho_1, \rho_2) \leq \frac{2e^{\frac{2R_0^2}{\epsilon} - \frac{3}{2}}}{\epsilon} W_2(\rho_1, \rho_2) \, ,
\end{aligned}
\end{equation} 
where the first inequality applies~\eqref{eqn:lem_reg_KL_grad_Lip_third_second_Lip} and Corollary~\ref{cor:kant_rub_bd}, and the second inequality uses~\ref{lem:W1_W2_order}.
Combining~\eqref{eqn:lem_reg_KL_grad_Lip_third_main},~\eqref{eqn:lem_reg_KL_grad_Lip_third_first}, and~\eqref{eqn:lem_reg_KL_grad_Lip_third_second} completes the proof of this lemma.
\end{proof}

\begin{lemma}\label{lem:reg_KL_grad_Lip_fourth}
Under the setup of Lemma~\ref{lem:reg_reg_KL_first_var}, 
\[
\left\Vert \frac{\int_{\mathcal{C}} \nabla \varphi_\epsilon(x-z)d\rho_1(z)}{\int_{\mathcal{C}} \varphi_\epsilon(x-w)d\rho_1(w)}  - \frac{\int_{\mathcal{C}} \nabla \varphi_\epsilon(y-z)d\rho_1(z)}{\int_{\mathcal{C}} \varphi_\epsilon(y-w)d\rho_1(w)}  \right\Vert_2 \leq \frac{e^{\frac{4R_0^2}{\epsilon} - 1} + 2e^{\frac{2R_0^2}{\epsilon} - \frac{3}{2}}}{\epsilon} \left\Vert x - y \right\Vert_2 \, .
\]
\end{lemma}

\begin{proof}
Note 
\begin{equation}\label{eqn:lem_reg_KL_grad_Lip_fourth_main}
\begin{aligned}
& \left\Vert \frac{\int_{\mathcal{C}} \nabla \varphi_\epsilon(x-z)d\rho_1(z)}{\int_{\mathcal{C}} \varphi_\epsilon(x-w)d\rho_1(w)}  - \frac{\int_{\mathcal{C}} \nabla \varphi_\epsilon(y-z)d\rho_1(z)}{\int_{\mathcal{C}} \varphi_\epsilon(y-w)d\rho_1(w)}  \right\Vert_2 \\
\leq & \left\Vert \frac{\int_{\mathcal{C}} \nabla \varphi_\epsilon(x-z)d\rho_1(z)}{\int_{\mathcal{C}} \varphi_\epsilon(x-w)d\rho_1(w)}  - \frac{\int_{\mathcal{C}} \nabla \varphi_\epsilon(x-z)d\rho_1(z)}{\int_{\mathcal{C}} \varphi_\epsilon(y-w)d\rho_1(w)}  \right\Vert_2 \\
& + \left\Vert \frac{\int_{\mathcal{C}} \nabla \varphi_\epsilon(x-z)d\rho_1(z)}{\int_{\mathcal{C}} \varphi_\epsilon(y-w)d\rho_1(w)}  - \frac{\int_{\mathcal{C}} \nabla \varphi_\epsilon(y-z)d\rho_1(z)}{\int_{\mathcal{C}} \varphi_\epsilon(y-w)d\rho_1(w)}  \right\Vert_2 \\
\leq & \left\vert \frac{\int_{\mathcal{C}} \left( \varphi_\epsilon(x-w) -  \varphi_\epsilon(y-w) \right)d\rho_1(w)}{\left(\int_{\mathcal{C}} \varphi_\epsilon(x-w)d\rho_1(w)\right)\left( \int_{\mathcal{C}} \varphi_\epsilon(y-w)d\rho_1(w) \right)} \right\vert \cdot \left\Vert\int_{\mathcal{C}} \nabla \varphi_\epsilon(x-z)d\rho_1(z)  \right\Vert_2 \\
& + \left\Vert \frac{\int_{\mathcal{C}} \left( \nabla \varphi_\epsilon(x-z) - \nabla \varphi_\epsilon(y-z) \right) d\rho_1(z)}{\int_{\mathcal{C}} \varphi_\epsilon(y-w)d\rho_1(w)}  \right\Vert_2 \\
\leq & \frac{e^{\frac{4R_0^2}{\epsilon} - \frac{1}{2}}}{\sqrt{\epsilon}}\left\vert \int_{\mathcal{C}} \left( \varphi_\epsilon(x-w) -  \varphi_\epsilon(y-w) \right)d\rho_1(w) \right\vert \\
& + e^{\frac{2R_0^2}{\epsilon}}\left\Vert \int_{\mathcal{C}} \left( \nabla \varphi_\epsilon(x-z) - \nabla \varphi_\epsilon(y-z) \right) d\rho_1(z)  \right\Vert_2 \\
\leq & \frac{e^{\frac{4R_0^2}{\epsilon} - 1}}{\epsilon} \left\Vert (x - w) - (y - w) \right\Vert_2 + \frac{2e^{\frac{2R_0^2}{\epsilon} - \frac{3}{2}}}{\epsilon}\left\Vert  (x - z) - (y - z) \right\Vert_2 \\
\leq & \frac{e^{\frac{4R_0^2}{\epsilon} - 1} + 2e^{\frac{2R_0^2}{\epsilon} - \frac{3}{2}}}{\epsilon} \left\Vert x - y \right\Vert_2 \, ,
\end{aligned}
\end{equation} where the first inequality applies triangular inequality; the third inequality uses Lemma~\ref{lem:mollifier_int_bounds}; and the fourth inequality utilizes Lemma~\ref{lem:mollifier_raw_bounds}.
\end{proof}

\begin{lemma}\label{lem:reg_KL_grad_Lip_fifth}
Under the setup of Lemma~\ref{lem:reg_reg_KL_first_var}, 
\[
\left\Vert \frac{\int_{\mathcal{C}} \nabla \varphi_\epsilon(y-z)d\rho_1(z)}{\int_{\mathcal{C}} \varphi_\epsilon(y-w)d\rho_1(w)}  - \frac{\int_{\mathcal{C}} \nabla \varphi_\epsilon(y-z)d\rho_2(z)}{\int_{\mathcal{C}} \varphi_\epsilon(y-w)d\rho_2(w)}  \right\Vert_2 \leq \frac{e^{\frac{4R_0^2}{\epsilon}-1} + e^{\frac{4R_0^2}{\epsilon} - \frac{3}{2}} }{\epsilon} W_2(\rho_1, \rho_2) \, .
\]
\end{lemma}

\begin{proof}
Note \begin{equation}
\begin{aligned}
& \left\Vert \frac{\int_{\mathcal{C}} \nabla \varphi_\epsilon(y-z)d\rho_1(z)}{\int_{\mathcal{C}} \varphi_\epsilon(y-w)d\rho_1(w)}  - \frac{\int_{\mathcal{C}} \nabla \varphi_\epsilon(y-z)d\rho_2(z)}{\int_{\mathcal{C}} \varphi_\epsilon(y-w)d\rho_2(w)}  \right\Vert_2  \\
= & \left\Vert \frac{\left( \int_{\mathcal{C}} \varphi_\epsilon(y-w)d\rho_2(w) \right) \left( \int_{\mathcal{C}} \nabla \varphi_\epsilon(y-z)d\rho_1(z) \right) }{\left( \int_{\mathcal{C}} \varphi_\epsilon(y-w)d\rho_1(w) \right)  \left( \int_{\mathcal{C}} \varphi_\epsilon(y-w)d\rho_2(w) \right)}\right. \\
& \ \ \ \ \ \ - \left.\frac{\left( \int_{\mathcal{C}} \varphi_\epsilon(y-w)d\rho_1(w) \right) \left(\int_{\mathcal{C}} \nabla \varphi_\epsilon(y-z)d\rho_2(z) \right) }{\left( \int_{\mathcal{C}} \varphi_\epsilon(y-w)d\rho_1(w) \right)  \left( \int_{\mathcal{C}} \varphi_\epsilon(y-w)d\rho_2(w) \right)}\right\Vert_2 \\
\leq & e^{\frac{4R_0^2}{\epsilon}} \left\Vert \left( \int_{\mathcal{C}} \varphi_\epsilon(y-w)d\rho_2(w) \right) \left( \int_{\mathcal{C}} \nabla \varphi_\epsilon(y-z)d\rho_1(z) \right) \right. \\
&  \ \ \ \ \ \ - \left.\left( \int_{\mathcal{C}} \varphi_\epsilon(y-w)d\rho_1(w) \right) \left(\int_{\mathcal{C}} \nabla \varphi_\epsilon(y-z)d\rho_2(z) \right) \right\Vert_2 \\
\leq & e^{\frac{4R_0^2}{\epsilon}} \left\vert \int_{\mathcal{C}} \varphi_\epsilon(y-w)d\rho_2(w)  -  \int_{\mathcal{C}} \varphi_\epsilon(y-w)d\rho_1(w)  \right\vert \cdot \left\Vert \int_{\mathcal{C}} \nabla \varphi_\epsilon(y-z)d\rho_1(z)  \right\Vert_2 \\
 & + e^{\frac{4R_0^2}{\epsilon}} \left\vert \int_{\mathcal{C}} \varphi_\epsilon(y-w)d\rho_1(w) \right\vert \cdot \left\Vert  \int_{\mathcal{C}} \nabla \varphi_\epsilon(y-z)d\rho_1(z)  - \int_{\mathcal{C}} \nabla \varphi_\epsilon(y-z)d\rho_2(z) \right\Vert_2 \\
 \leq & \frac{e^{\frac{4R_0^2}{\epsilon}-\frac{1}{2}}}{\sqrt{\epsilon}} \left\vert \int_{\mathcal{C}} \varphi_\epsilon(y-w)d\rho_2(w)  -  \int_{\mathcal{C}} \varphi_\epsilon(y-w)d\rho_1(w)  \right\vert \\
 & + e^{\frac{4R_0^2}{\epsilon}} \left\Vert  \int_{\mathcal{C}} \nabla \varphi_\epsilon(y-z)d\rho_1(z)  - \int_{\mathcal{C}} \nabla \varphi_\epsilon(y-z)d\rho_2(z) \right\Vert_2 \\
 \leq & \frac{e^{\frac{4R_0^2}{\epsilon}-1}}{\epsilon} W_1(\rho_1, \rho_2) + \frac{e^{\frac{4R_0^2}{\epsilon} - \frac{3}{2}}}{\epsilon} W_1(\rho_1, \rho_2) \\
 \leq & \frac{e^{\frac{4R_0^2}{\epsilon}-1} + e^{\frac{4R_0^2}{\epsilon} - \frac{3}{2}} }{\epsilon} W_2(\rho_1, \rho_2) \, ,
\end{aligned}
\end{equation}
where the first and third inequalities use Lemma~\ref{lem:mollifier_int_bounds}; the second inequality applies triangle inequality; the fourth inequality applies Lemma~\ref{lem:mollifier_raw_bounds} and Corollary~\ref{cor:kant_rub_bd}; and the last inequality utilizes~\ref{lem:W1_W2_order}.
\end{proof}

\subsection{Details for~\eqref{eqn:reg_KL_grad_transport_Lip}}

\begin{lemma}\label{lem:reg_KL_W2_Lip_first}
Under the setup of Lemma~\ref{lem:reg_reg_KL_first_var}, $$\frac{\sqrt{\int_{\mathcal{C}} \left\Vert \int_{\mathbb{R}^d} \frac{\nabla \varphi_\epsilon(x-z)}{\int_{\mathcal{C}}\varphi_\epsilon(z - w) d \rho_1(w)} d \rho_1 (z) - \int_{\mathcal{C}} \frac{\nabla \varphi_\epsilon(y-z)}{\int_{\mathcal{C}}\varphi_\epsilon(z - w) d \rho_2(w)} d \rho_2 (z) \right\Vert_2^2 d \pi(x,y) }}{\sqrt{\int_{\mathcal{C}} \left\Vert x - y \right\Vert_2^2 d \pi(x,y)}} \leq \frac{4e^{\frac{4R_0^2}{\epsilon}-\frac{3}{2}} + 2 e^{\frac{8R_0^2}{\epsilon} - 1}}{\epsilon}  \, .$$
\end{lemma}

\begin{proof}
First note \begin{equation}\label{eqn:reg_KL_W2_Lip_first_main}
\begin{aligned}
& \int_{\mathcal{C}} \frac{\nabla \varphi_\epsilon(x  - z)}{\int_{\mathcal{C}}\varphi_\epsilon(z - w) d  \rho_1(w)} d  \rho_1 (z) - \int_{\mathcal{C}} \frac{\nabla \varphi_\epsilon(y-z)}{\int_{\mathcal{C}}\varphi_\epsilon(z - w) d \rho_2(w)} d \rho_2 (z) \\
= & \int_{\mathcal{C}^2} \frac{\nabla \varphi_\epsilon(x - p)}{\int_{\mathcal{C}^2}\varphi_\epsilon(p - s) d \pi(s,t)} d \pi (p,q) - \int_{\mathcal{C}^2} \frac{\nabla \varphi_\epsilon(y - q)}{\int_{\mathcal{C}^2}\varphi_\epsilon(q - t) d \pi(s,t)} d \pi(p,q) \\
= & \int_{\mathcal{C}^2} \frac{\nabla \varphi_\epsilon(x - p)}{\int_{\mathcal{C}^2}\varphi_\epsilon(p - s) d \pi(s,t)} d \pi (p,q) - \int_{\mathcal{C}^2} \frac{\nabla \varphi_\epsilon(y - q)}{\int_{\mathcal{C}^2}\varphi_\epsilon(p - s) d \pi(s,t)} d \pi(p,q) \\
& + \int_{\mathcal{C}^2} \frac{\nabla \varphi_\epsilon(y - q)}{\int_{\mathcal{C}^2}\varphi_\epsilon(p - s) d \pi(s,t)} d \pi (p,q) - \int_{\mathcal{C}^2} \frac{\nabla \varphi_\epsilon(y - q)}{\int_{\mathcal{C}^2}\varphi_\epsilon(q - t) d \pi(s,t)} d \pi(p,q) \, .
\end{aligned}
\end{equation}

For the first part, \begin{equation}\label{eqn:reg_KL_W2_Lip_first_first}
\begin{aligned}
& \int_{\mathcal{C}^2} \left\Vert \int_{\mathcal{C}^2} \frac{\nabla \varphi_\epsilon(x - p)}{\int_{\mathcal{C}^2}\varphi_\epsilon(p - s) d \pi(s,t)} d \pi (p,q) - \int_{\mathcal{C}^2} \frac{\nabla \varphi_\epsilon(y - q)}{\int_{\mathcal{C}^2}\varphi_\epsilon(p - s) d \pi(s,t)} d \pi(p,q) \right\Vert^2 d \pi(x, y) \\
= & \int_{\mathcal{C}^2} \left\Vert \int_{\mathcal{C}^2} \frac{\nabla \varphi_\epsilon(x - p) - \nabla \varphi_\epsilon(y - q)}{\int_{\mathcal{C}^2}\varphi_\epsilon(p - s) d \pi(s,t)} d \pi (p,q) \right\Vert^2 d \pi(x, y) \\
\leq & e^{\frac{8R_0^2}{\epsilon}} \int_{\mathcal{C}^2} \left\Vert \int_{\mathcal{C}^2} \nabla \varphi_\epsilon(x - p) - \nabla \varphi_\epsilon(y - q) d \pi (p,q) \right\Vert^2 d \pi(x, y) \\
\leq & e^{\frac{8R_0^2}{\epsilon}}  \int_{\mathcal{C}^2}  \int_{\mathcal{C}^2} \left\Vert \nabla \varphi_\epsilon(x - p) - \nabla \varphi_\epsilon(y - q) \right\Vert^2 d \pi (p,q)  d \pi(x, y) \\
\leq & e^{\frac{8R_0^2}{\epsilon}}  \int_{\mathcal{C}^2}  \int_{\mathcal{C}^2} \frac{4}{\epsilon^2} e^{-3}  \left\Vert \left( x - p \right) - \left( y - q \right) \right\Vert^2 d \pi (p,q) d \pi(x, y) \\
\leq & e^{\frac{8R_0^2}{\epsilon}}  \int_{\mathcal{C}^2}  \int_{\mathcal{C}^2} \frac{8}{\epsilon^2} e^{-3}  \left(\left\Vert  x - y \right\Vert^2 + \left\Vert p - q \right\Vert^2 \right) d \pi (p,q) d \pi(x, y) \\
= & e^{\frac{8R_0^2}{\epsilon}}  \int_{\mathcal{C}^2}  \frac{16}{\epsilon^2} e^{-3}  \left\Vert  x - y \right\Vert^2  d \pi(x, y) \\
\leq & \frac{16}{\epsilon^2}  e^{\frac{8 R_0^2}{\epsilon}- 3} \int_{\mathcal{C}^2}   \left\Vert  x - y \right\Vert^2  d \pi(x, y) \, ,
\end{aligned}
\end{equation} where in the first inequality we used Lemma~\ref{lem:mollifier_int_bounds}, and in the third inequality we applied Lemma~\ref{lem:mollifier_raw_bounds}.

For the second part, \begin{equation}\label{eqn:reg_KL_W2_Lip_first_second}
\begin{aligned}
& \int_{\mathcal{C}^2} \left\Vert \int_{\mathcal{C}^2} \frac{\nabla \varphi_\epsilon(y - q)}{\int_{\mathcal{C}^2}\varphi_\epsilon(p - s) d \pi(s,t)} d \pi (p,q) - \int_{\mathcal{C}^2} \frac{\nabla \varphi_\epsilon(y - q)}{\int_{\mathcal{C}^2}\varphi_\epsilon(q - t) d \pi(s,t)} d \pi(p,q) \right\Vert^2 d \pi(x,y) \\
= & \int_{\mathcal{C}^2} \left\Vert \int_{\mathcal{C}^2} \frac{\int_{\mathcal{C}^2}\left( \varphi_\epsilon(p - s) - \varphi_\epsilon(q - t) \right) d \pi(s,t)}{\left(\int_{\mathcal{C}^2}\varphi_\epsilon(p - s) d \pi(s,t)\right)\left( \int_{\mathcal{C}^2}\varphi_\epsilon(q - t) d \pi(s,t)\right)} \nabla \varphi_\epsilon(y - q) d \pi(p,q) \right\Vert^2 d \pi(x,y) \\
\leq & \int_{\mathcal{C}^2}  \int_{\mathcal{C}^2} \left\Vert \frac{\int_{\mathcal{C}^2}\left( \varphi_\epsilon(p - s) - \varphi_\epsilon(q - t) \right) d \pi(s,t)}{\left(\int_{\mathcal{C}^2}\varphi_\epsilon(p - s) d \pi(s,t)\right)\left( \int_{\mathcal{C}^2}\varphi_\epsilon(q - t) d \pi(s,t)\right)} \nabla \varphi_\epsilon(y - q) \right\Vert^2 d \pi(p,q)  d \pi(x,y) \\
\leq & e^{\frac{16 R_0^2}{\epsilon}} \int_{\mathcal{C}^2}  \int_{\mathcal{C}^2} \left\Vert \int_{\mathcal{C}^2}\left( \varphi_\epsilon(p - s) - \varphi_\epsilon(q - t) \right) d \pi(s,t) \nabla \varphi_\epsilon(y - q) \right\Vert^2 d \pi(p,q)  d \pi(x,y) \\
\leq & \frac{1}{\epsilon}e^{\frac{16 R_0^2}{\epsilon} - 1} \int_{\mathcal{C}^2}  \int_{\mathcal{C}^2} \left\vert \int_{\mathcal{C}^2}\left( \varphi_\epsilon(p - s) - \varphi_\epsilon(q - t) \right) d \pi(s,t)  \right\vert^2 d \pi(p,q)  d \pi(x,y) \\
\leq & \frac{1}{\epsilon}e^{\frac{16 R_0^2}{\epsilon} - 1} \int_{\mathcal{C}^2}  \int_{\mathcal{C}^2}  \int_{\mathcal{C}^2} \left\vert \varphi_\epsilon(p - s) - \varphi_\epsilon(q - t) \right\vert^2 d \pi(s,t)  d \pi(p,q)  d \pi(x,y) \\
\leq & \frac{1}{\epsilon}e^{\frac{16 R_0^2}{\epsilon} - 1} \frac{1}{\epsilon e}\int_{\mathcal{C}^2}  \int_{\mathcal{C}^2}  \int_{\mathcal{C}^2} \left\Vert \left(p - s\right) - \left(q - t\right) \right\Vert^2 d \pi(s,t)  d \pi(p,q)  d \pi(x,y) \\
\leq & \frac{4}{\epsilon^2}e^{\frac{16 R_0^2}{\epsilon} - 2} \int_{\mathcal{C}^2}  \left\Vert x - y \right\Vert^2  d \pi(x,y) \, ,
\end{aligned}
\end{equation} 
where in the second inequality we used Lemma~\ref{lem:mollifier_int_bounds}, and in the fifth inequality we applied Lemma~\ref{lem:mollifier_raw_bounds}.

Combining~\eqref{eqn:reg_KL_W2_Lip_first_main},~\eqref{eqn:reg_KL_W2_Lip_first_first} and~\eqref{eqn:reg_KL_W2_Lip_first_second} completes the proof of the lemma.
\end{proof}

\begin{lemma}\label{lem:reg_KL_W2_Lip_second}
Under the setup of Lemma~\ref{lem:reg_reg_KL_first_var}, $$\frac{\sqrt{\int_{\mathcal{C}} \left\Vert \frac{\int_{\mathcal{C}} \nabla \varphi_\epsilon(x-z)d\rho_1(z)}{\int_{\mathcal{C}} \varphi_\epsilon(x-w)d\rho_1(w)} - \frac{\int_{\mathcal{C}} \nabla \varphi_\epsilon(y-z)d\rho_2(z)}{\int_{\mathcal{C}} \varphi_\epsilon(y-w)d\rho_2(w)} \right\Vert_2^2 d \pi(x,y)}}{\sqrt{\int_{\mathcal{C}} \left\Vert x - y \right\Vert_2^2 d \pi(x,y)}} \leq \frac{4e^{\frac{4 R_0^2}{\epsilon}-\frac{3}{2}} + 2 e^{\frac{8 R_0^2}{\epsilon} - 2}}{\epsilon}  \, .$$
\end{lemma}

\begin{proof}
First note \begin{equation}\label{eqn:reg_KL_W2_Lip_second_main}
\begin{aligned}
& \frac{\int_{\mathcal{C}} \nabla \varphi_\epsilon(x-z)d\rho_1(z)}{\int_{\mathcal{C}} \varphi_\epsilon(x-w)d\rho_1(w)} - \frac{\int_{\mathcal{C}} \nabla \varphi_\epsilon(y-z)d\rho_2(z)}{\int_{\mathcal{C}} \varphi_\epsilon(y-w)d\rho_2(w)} \\
= & \frac{\int_{\mathcal{C}^2} \nabla \varphi_\epsilon(x-p)d\pi(p,q)}{\int_{\mathcal{C}^2} \varphi_\epsilon(x-s)d\pi(s,t)} - \frac{\int_{\mathcal{C}^2} \nabla \varphi_\epsilon(y-q)d\pi(p,q)}{\int_{\mathcal{C}^2} \varphi_\epsilon(y-t)d\pi(s,t)} \\
= & \frac{\int_{\mathcal{C}^2} \left(\nabla \varphi_\epsilon(x-p) - \nabla \varphi_\epsilon(y-q) \right)d\pi(p,q)}{\int_{\mathcal{C}^2} \varphi_\epsilon(x-s)d\pi(s,t)}  \\
& +  \frac{\int_{\mathcal{C}^2} \nabla \varphi_\epsilon(y - q)d\pi(p,q)}{\int_{\mathcal{C}^2} \varphi_\epsilon(x-s)d\pi(s,t)} - \frac{\int_{\mathcal{C}^2} \nabla \varphi_\epsilon(y-q)d\pi(p,q)}{\int_{\mathcal{C}^2} \varphi_\epsilon(y-t)d\pi(s,t)} \, .
\end{aligned}
\end{equation}

For the first part, \begin{equation}\label{eqn:reg_KL_W2_Lip_second_first}
\begin{aligned}
& \int_{\mathcal{C}^2} \left\Vert \frac{\int_{\mathcal{C}^2} \left(\nabla \varphi_\epsilon(x-p) - \nabla \varphi_\epsilon(y-q) \right)d\pi(p,q)}{\int_{\mathcal{C}^2} \varphi_\epsilon(x-s)d\pi(s,t)} \right\Vert^2  d \pi(x,y) \\
\leq & e^{\frac{8 R_0^2}{\epsilon}} \int_{\mathcal{C}^2} \left\Vert \int_{\mathcal{C}^2} \left(\nabla \varphi_\epsilon(x-p) - \nabla \varphi_\epsilon(y-q) \right)d\pi(p,q) \right\Vert^2  d \pi(x,y) \\
\leq & e^{\frac{8 R_0^2}{\epsilon}} \int_{\mathcal{C}^2} \int_{\mathcal{C}^2} \left\Vert \nabla \varphi_\epsilon(x-p) - \nabla \varphi_\epsilon(y-q) \right\Vert^2 d\pi(p,q) d \pi(x,y) \\
\leq & \frac{4}{\epsilon^2} e^{\frac{8 R_0^2}{\epsilon} - 3} \int_{\mathcal{C}^2} \int_{\mathcal{C}^2} \left\Vert \left(x-p\right) - \left(y-q\right) \right\Vert^2 d\pi(p,q) d \pi(x,y) \\
\leq & \frac{16}{\epsilon^2} e^{\frac{8 R_0^2}{\epsilon} - 3} \int_{\mathcal{C}^2} \left\Vert x - y \right\Vert^2 d \pi(x,y) \, ,
\end{aligned}
\end{equation} where in the first inequality we used Lemma~\ref{lem:mollifier_int_bounds}, and in the third inequality we applied Lemma~\ref{lem:mollifier_raw_bounds}.

For the second part, \begin{equation}\label{eqn:reg_KL_W2_Lip_second_second}
\begin{aligned}
& \int_{\mathcal{C}^2} \left\Vert \frac{\int_{\mathcal{C}^2} \nabla \varphi_\epsilon(y - q)d\pi(p,q)}{\int_{\mathcal{C}^2} \varphi_\epsilon(x-s)d\pi(s,t)} - \frac{\int_{\mathcal{C}^2} \nabla \varphi_\epsilon(y-q)d\pi(p,q)}{\int_{\mathcal{C}^2} \varphi_\epsilon(y-t)d\pi(s,t)}  \right\Vert^2  d \pi(x,y) \\
= & \int_{\mathcal{C}^2} \left\Vert \left(\frac{1}{\int_{\mathcal{C}^2} \varphi_\epsilon(x-s)d\pi(s,t)} - \frac{1}{\int_{\mathcal{C}^2} \varphi_\epsilon(y-t)d\pi(s,t)} \right) \left( \int_{\mathcal{C}^2} \nabla \varphi_\epsilon(y - q)d\pi(p,q) \right) \right\Vert^2  d \pi(x,y) \\
\leq & \frac{1}{\epsilon}e^{-1} \int_{\mathcal{C}^2} \left\vert \frac{1}{\int_{\mathcal{C}^2} \varphi_\epsilon(x-s)d\pi(s,t)} - \frac{1}{\int_{\mathcal{C}^2} \varphi_\epsilon(y-t)d\pi(s,t)}  \right\vert^2  d \pi(x,y) \\
\leq & \frac{1}{\epsilon}e^{-1} \int_{\mathcal{C}^2} \left\vert \frac{\int_{\mathcal{C}^2} \left( \varphi_\epsilon(x-s) - \varphi_\epsilon(y-t) \right) d\pi(s,t)}{\left(\int_{\mathcal{C}^2} \varphi_\epsilon(x-s)d\pi(s,t)\right) \left( \int_{\mathcal{C}^2} \varphi_\epsilon(y-t)d\pi(s,t) \right)}  \right\vert^2  d \pi(x,y) \\
\leq & \frac{1}{\epsilon}e^{\frac{16 R_0^2}{\epsilon} - 3} \int_{\mathcal{C}^2} \left\vert \int_{\mathcal{C}^2} \left( \varphi_\epsilon(x-s) - \varphi_\epsilon(y-t) \right) d\pi(s,t)  \right\vert^2  d \pi(x,y) \\
\leq & \frac{1}{\epsilon}e^{\frac{16 R_0^2}{\epsilon} - 3} \int_{\mathcal{C}^2}  \int_{\mathcal{C}^2} \left\vert \varphi_\epsilon(x-s) - \varphi_\epsilon(y-t)  \right\vert^2 d\pi(s,t)   d \pi(x,y) \\
\leq & \frac{1}{\epsilon^2}e^{\frac{16 R_0^2}{\epsilon} - 4} \int_{\mathcal{C}^2}  \int_{\mathcal{C}^2} \left\vert \left( x - s \right) - \left( y - t \right)  \right\vert^2 d\pi(s,t)   d \pi(x,y) \\
\leq & \frac{4}{\epsilon^2}e^{\frac{16 R_0^2}{\epsilon} - 4} \int_{\mathcal{C}^2}  \left\Vert x - y  \right\Vert^2  d \pi(x,y) \, ,
\end{aligned}
\end{equation} where in the first inequality we used Lemma~\ref{lem:mollifier_raw_bounds}, and in the third inequality we applied Lemma~\ref{lem:mollifier_int_bounds}.

Combining~\eqref{eqn:reg_KL_W2_Lip_second_main},~\eqref{eqn:reg_KL_W2_Lip_second_first} and~\eqref{eqn:reg_KL_W2_Lip_second_second} completes the proof of the lemma.
\end{proof}

\section{Convergence of PGD in Hilbert Spaces}\label{sec:Supplement_proof_prop_W2_PGD_rate}

In this section we establish several convergence results for Projected Gradient Descent (PGD) in Hilbert spaces, which are used in the proof of Proposition~\ref{prop:W2_PGD_rate}. 
These statements are direct Hilbert-space adaptations of well-known results in the Euclidean setting (see, for instance,~\cite{WR22}). We include them here for completeness and clarity, noting that they are fairly standard and logically independent of the main text.

Throughout this section, let $(\mathcal H,\langle\cdot,\cdot\rangle)$ be a real Hilbert space, let $\mathcal{K}$ be a nonempty convex closed subset of $\mathcal{H}$, and let $\operatorname{proj}_C : \mathcal{H} \to \mathcal{K}$ be the projection operator onto $\mathcal{K}$.
Let $F:\mathcal{H} \to \mathbb{R}$ be Fr\'echet differentiable, and let $G[X] \in \mathcal{H}$ denote the Hilbert-space gradient of $F$ at $X$.
Let $F_\mathrm{opt} = \inf_{X \in \mathcal{K}}F[X]$, and $X_\mathrm{opt} \in \mathcal{K}$ satisfies $X_\mathrm{opt} = F_\mathrm{opt}$.

We first recall a few standard notions in Hilbert spaces, which are direct generalizations of the Euclidean space cases:

\begin{itemize}
\item[--]\textbf{$L$-smoothness:}  
$F$ is said to be $L$-smooth on $\mathcal{K}$ if its gradient $G$ is $L$-Lipschitz on $\mathcal{K}$, namely
\[
\|G[X]-G[Y]\| \;\leq\; L\|X-Y\|, \qquad \forall X,Y\in\mathcal{K}.
\]
Equivalently, for all $X,Y\in\mathcal{K}$,
\[\begin{aligned}
F(Y) &\;\geq\; F(X) + \langle G[X],Y-X\rangle - \tfrac{L}{2}\|Y-X\|^2 \, , \\
F(Y) &\;\leq\; F(X) + \langle G[X],Y-X\rangle + \tfrac{L}{2}\|Y-X\|^2 \, .
\end{aligned}\]

\item[--]\textbf{Convexity:}  
$F$ is convex on $\mathcal{K}$ if
\[
F(Y) \;\geq\; F(X) + \langle G[X],Y-X\rangle, \qquad \forall X,Y\in\mathcal{K}.
\]

\item[--]\textbf{Strong convexity:}  
$F$ is $m$-strongly convex on $\mathcal{K}$ for some $m>0$ if
\[
F(Y) \;\geq\; F(X) + \langle G[X],Y-X\rangle 
+ \tfrac{m}{2}\|Y-X\|^2, 
\qquad \forall X,Y\in\mathcal{K}.
\]

\item[--]\textbf{Normal cone:}
At any $X \in \mathcal{K}$, the normal cone of $\mathcal{K}$ at $X$ is \[\operatorname{N}_{\mathcal{K}}(X) := \left\{ Z \in \mathcal{H} \quad : \quad \left\langle Z , Y - X \right\rangle \leq 0 \quad \forall Y \in \mathcal{K} \right\}
\]
\end{itemize}

Let the PGD updates be defined as \begin{equation}\label{eqn:def_PGD_Hilbert_abstract}
X_{n+1} = \operatorname{proj}_{\mathcal{K}}\left( X_n - h G[X_n]\right) \, .
\end{equation}

We are now ready to state and prove the lemmas.

\begin{lemma}[First Order Optimality]\label{lem:Hilbert_local_optimality}
If $X^* \in \mathcal{K}$ is a local minimizer of $F$ on $\mathcal{K}$, then $-\;G(X^{*})\;\in\;\operatorname{N}_{\mathcal{K}}(X^{*})$.
\end{lemma}

\begin{proof}
Let $Z\in\mathcal{K}$ be arbitrary.  
For any $t\in[0,1]$ the point $X^{*} + t(Z-X^{*})$ lies in $\mathcal{K}$; hence, by local optimality, there exists some $\alpha \in (0,1)$ such thhat for any $t \in [0,\alpha]$:
\[
F\bigl(X^{*}+t (Z-X^{*})\bigr)\;\geq\;F(X^{*})
 \, .
\]
Since $F$ is differentiable, a Taylor expansion along the line
$t\mapsto X^{*}+t(Z-X^{*})$ gives
\[
F\bigl(X^{*}+t(Z-X^{*})\bigr)
= F(X^{*}) + t\,\langle G(X^{*}),\,Z-X^{*}\rangle + o(t) \, .
\]
Substituting the inequality above and dividing by $t>0$ yields that for any $Z \in \mathcal{K}$
\[
\langle -G(X^{*}),\,Z-X^{*}\rangle\;\leq\; 0 \, .
\]
This precisely means that $-G(X^{*})\in\operatorname{N}_{\mathcal{K}}(X^{*})$ by the definition of the normal cone.
\end{proof}

\begin{lemma}[Nonconvex Rate to Stationary]\label{lem:nonconvex_Hilbert_rate}
Let $F$ be $L$-smooth on $\mathcal{K}$ and $h \in (0, \frac{1}{L}]$.
Then $F(X_{n}) \leq F(X_{n+1})$ and \[
\min_{0 \leq n \leq T-1} \| X_{n+1} - X_n \| \leq \sqrt{\frac{2 h \left( F(X_0) - F_\mathrm{opt} \right)}{T}} \, .
\]
\end{lemma}

\begin{proof}
Define $Q_n : \mathcal{K} \to \mathbb{R}$ as \[
Q_n(X) := F(X_n) + \left\langle G(X_n), X - X_n\right\rangle + \frac{1}{2h}\|X - X_n\|^2 \, . 
\]
Then \[
\operatorname{argmin}_{X \in \mathcal{K}}Q_n(X) = \operatorname{proj}_{\mathcal{K}}\left(X_n - h G(X_n) \right) = X_{n+1} \, .
\]
By Lemma~\ref{lem:Hilbert_local_optimality}, the gradient of $Q_n$ at $X_{n+1}$ is in the normal cone of $\mathcal{K}$ at $X_{n+1}$, which gives \begin{equation}\label{eqn:PGD_Hilbert_optimality}
- G(X_n) - \frac{1}{h} \left( X_{n+1} - X_n \right) 
\in \operatorname{N}_{\mathcal{K}}(X_{n+1}) \, ,
\end{equation} 
This further implies \[
\left\langle - G(X_n) - \frac{1}{h} \left( X_{n+1} - X_n \right) , X_n - X_{n+1}  \right\rangle \leq 0 \, .
\]
Rearranging the term yields \[
\left\langle G(X_n) , X_n - X_{n+1} \right\rangle  \geq \frac{1}{h} \| X_n - X_{n+1} \|^2 \, . 
\]
Since $F(X_n) = Q_n(X_n)$ and $F(X_{n+1}) \leq Q_n(X_{n+1})$, \[\begin{aligned}
F(X_{n}) - F(X_{n+1}) & \geq Q_n(X_{n}) - Q_n(X_{n+1}) \\
& = \left\langle G(X_n), X_n - X_{n+1}\right\rangle - \frac{1}{2h} \| X_n - X_{n+1} \|^2 \\
& \geq \frac{1}{2h} \| X_{n+1} - X_{n} \|^2 \, .
\end{aligned}\]
This shows that $F(X_{n+1}) \leq F(X_n)$.
Taking the ergodic average gives us \[\begin{aligned}
\min_{0 \leq n \leq T-1} \| X_{n+1} - X_n \|^2 & \leq \frac{1}{T} \sum_{n=0}^{T-1}\| X_{n+1} - X_n \|^2 \\
& \leq \frac{1}{T}\sum_{n=0}^{T-1} 2h \left(F(X_n) - F(X_{n+1})  \right) \\
& = \frac{2h}{T} \left( F(X_0) - F(X_T) \right) \\
& \leq \frac{2h}{T} \left( F(X_0) - F(X_\mathrm{opt}) \right) \, .
\end{aligned}\]
The proof is now complete after taking the square root.
\end{proof}

\begin{lemma}[Convex Rate of Convergence]\label{lem:convex_Hilbert_rate}
Let $F$ be $L$-smooth and convex on $\mathcal{K}$ and $h \in (0, \frac{1}{L}]$.
Then \[
F(X_{n}) - F_\mathrm{opt} \leq \frac{\left\| X_{0} - X_\mathrm{opt} \right\|^2}{2 n h} \, .
\]
\end{lemma}

\begin{proof}
By convexity, $F(X_\mathrm{opt}) \geq F(X_n) + \left\langle G(X_n), X_\mathrm{opt} - X_n \right\rangle$, and therefore \begin{equation}\label{eqn:convex_Hilbert_estim}
\begin{aligned}
& F(X_{n+1}) - F(X_\mathrm{opt}) \\ \leq & F(X_{n+1}) - F(X_n) - \left\langle G(X_n) , X_\mathrm{opt} - X_n \right\rangle \\
= & F(X_{n+1}) - F(X_n) - \left\langle G(X_n) , X_{n+1} - X_n \right\rangle + \left\langle G(X_n) ,  X_{n+1} - X_\mathrm{opt} \right\rangle \, .
\end{aligned} 
\end{equation}
By~\eqref{eqn:PGD_Hilbert_optimality}, \[
\left\langle - G(X_n) - \frac{1}{h} \left( X_{n+1} - X_n \right) , X_\mathrm{opt} - X_{n+1}  \right\rangle \leq 0 \, .
\]
Rearranging the terms gives \[\begin{aligned}
& \left\langle G(X_{n}) , X_{n+1} - X_\mathrm{opt} \right\rangle \\
\leq & \frac{1}{h} \left\langle X_{n+1} - X_n , X_\mathrm{opt} - X_{n+1} \right\rangle \\
= & \frac{1}{2h} \left( \| X_{n} - X_\mathrm{opt} \|^2 - \| X_{n+1} - X_\mathrm{opt} \|^2 - \| X_{n+1} - X_\mathrm{n} \|^2 \right) \, .
\end{aligned}\]
Plug this into~\eqref{eqn:convex_Hilbert_estim}, we get \[\begin{aligned}
& F(X_{n+1}) - F(X_\mathrm{opt}) \\
\leq & F(X_{n+1}) - F(X_n) - \left\langle G(X_n) , X_{n+1} - X_n \right\rangle - \frac{1}{2h}\left\| X_{n+1} - X_\mathrm{n} \right\|^2 \\
&  + \frac{1}{2h}\left\| X_{n} - X_\mathrm{opt} \right\|^2 - \frac{1}{2h}\left\| X_{n+1} - X_\mathrm{opt} \right\|^2 \\
\leq & \frac{1}{2h}\left\| X_{n} - X_\mathrm{opt} \right\|^2 - \frac{1}{2h}\left\| X_{n+1} - X_\mathrm{opt} \right\|^2 \, ,
\end{aligned}\]
where the second inequality follows from the $L$-smoothness of $F$ and the fact that $h \in (0,\frac{1}{L}]$.
By taking the ergodic average, we obtain \[\begin{aligned}
\frac{1}{n}\sum_{k=0}^{n-1} \left(F(X_{k}) - F(X_\mathrm{opt})\right) & \leq \frac{1}{n}\sum_{n=0}^{n-1} \left( \frac{1}{2h}\left\| X_{k} - X_\mathrm{opt} \right\|^2 - \frac{1}{2h}\left\| X_{k+1} - X_\mathrm{opt} \right\|^2\right) \\
& \leq \frac{1}{2 n h} \left\| X_{0} - X_\mathrm{opt} \right\|^2 - \frac{1}{2 n h} \left\| X_{n} - X_\mathrm{opt} \right\|^2 \\
& \leq \frac{1}{2 n h} \left\| X_{0} - X_\mathrm{opt} \right\|^2 \, .
\end{aligned}\]
Recall that in Lemma~\ref{lem:nonconvex_Hilbert_rate} we showed that $F(X_{n+1}) \leq F(X_n)$, and therefore \[
F(X_{n}) - F_\mathrm{opt} \leq \frac{1}{n}\sum_{k=0}^{n-1} \left(F(X_{k}) - F(X_\mathrm{opt})\right) \leq \frac{\left\| X_{0} - X_\mathrm{opt} \right\|^2}{2 n h} \, .
\]
\end{proof}

\begin{lemma}[Strongly-Convex Rate of Convergence]\label{lem:strongly_convex_Hilbert_rate}
Let $F$ be $L$-smooth and $m$-strongly convex on $\mathcal{K}$ for $m > 0$ and $h \in (0, \frac{1}{L}]$.
Then \[
\left\| X_{n} - X_\mathrm{opt} \right\|^2 \leq \left( 1 - mh \right)^{n} \left\| X_{0} - X_\mathrm{opt} \right\|^2 \, .
\]
\end{lemma}

\begin{proof}
Since $F$ is $L$-smooth and $m$-strongly convex and $h \in (0, \frac{1}{K}]$, \[
\left\| G(X_n) - G(X_\mathrm{opt}) \right\|^2 \leq \frac{1}{h}\left\langle G(X_n) - G(X_\mathrm{opt}) , X_n - X_\mathrm{opt} \right\rangle \, .
\]
Also, \[\begin{aligned}
& \left\| X_{n+1} - X_\mathrm{opt} \right\|^2 \\
= & \left\| \operatorname{proj}_\mathcal{K} \left( X_n - h G(X_n)\right) - \operatorname{proj}_\mathcal{K} \left( X_\mathrm{opt} - h G(X_\mathrm{opt})\right) \right\| \\
\leq & \left\| \left( X_n - h G(X_n)\right) -  \left( X_\mathrm{opt} - h G(X_\mathrm{opt})\right) \right\| \\
= & \left\| X_n - X_\mathrm{opt} \right\|^2 + h^2 \left\| G(X_n) - G(X_\mathrm{opt})\right\|^2 - 2h \left\langle X_n - X_\mathrm{opt}, G(X_n) - G(X_\mathrm{opt})\right\rangle \, ,
\end{aligned}\] where the inequality used the nonexpansiveness of the projection operator.
Combining the two inequalities above, we obtain \begin{equation}\label{eqn:strongly_convex_Hilbert_descent_estimate}
\left\| X_{n+1} - X_\mathrm{opt} \right\|^2 \leq \left\| X_{n} - X_\mathrm{opt} \right\|^2 - h \left\langle X_n - X_\mathrm{opt}, G(X_n) - G(X_\mathrm{opt})\right\rangle \, .
\end{equation}
By the $m$-strong convexity of $F$, we have $m$-strong monotonicity \[
\left\langle X_n - X_\mathrm{opt}, G(X_n) - G(X_\mathrm{opt})\right\rangle \geq m \left\| X_n - X_\mathrm{opt} \right\|^2  \, .
\]
Plugging this back into~\eqref{eqn:strongly_convex_Hilbert_descent_estimate} gives us \[
\left\| X_{n+1} - X_\mathrm{opt} \right\|^2 \leq \left( 1 - mh \right)\left\| X_{n} - X_\mathrm{opt} \right\|^2 \, .
\]
Applying this inequality recursively completes the proof.
\end{proof}

\printbibliography

\end{document}